\documentclass[10pt]{amsart}

\usepackage{amsmath}
\usepackage{amssymb}
\usepackage{array,subfigure}
\usepackage{graphicx}

\newtheorem{thm}{Theorem}[section]
\newtheorem{prop}[thm]{Proposition}
\newtheorem{cor}[thm]{Corollary}
\newtheorem{lem}[thm]{Lemma}
\newtheorem{defn}[thm]{Definition}
\newtheorem{conj}[thm]{Conjecture}
\newtheorem{prob}[thm]{Problem}

\newenvironment{remark}{\refstepcounter{thm} \medskip \noindent {\bf  Remark \arabic{section}.\arabic{thm}.}}{\hfill\mbox{}\bigskip}

\newcounter{num}

\newenvironment{thmlist}{\begin{list}{(\roman{num})}{\usecounter{num}\setlength{\leftmargin}{25pt}
\setlength{\itemindent}{0pt}\setlength{\labelwidth}{20pt}\setlength{\labelsep}{5pt}\setlength{\itemsep}{0in}}}{\end{list}}

\newcommand{\C}{\mathbb{C}}

\newcommand{\R}{\mathbb{R}}
\newcommand{\Z}{\mathbb{Z}}
\newcommand{\N}{\mathbb{N}}
\newcommand{\Q}{\mathbb{Q}}
\newcommand{\cps}{\mathbb{C}P}

\newcommand{\ol}[1]{\bar{#1}}

\newcommand{\Aut}{\operatorname{Aut}}

\newcommand{\contr}{\,\lrcorner\,}

\newcommand{\dist}{\operatorname{dist}}

\newcommand{\Hol}{\operatorname{Hol}}
\newcommand{\Hom}{\operatorname{Hom}}
\newcommand{\im}{\operatorname{Im}}

\newcommand{\lcm}{\operatorname{lcm}}

\newcommand{\pic}{\operatorname{Pic}}

\newcommand{\ric}{\operatorname{Ricci}}
\newcommand{\Ric}{\operatorname{Ric}}

\title{K\"{a}hler-Einstein metrics on strictly pseudoconvex domains}
\author{Craig van Coevering}
\address{Department of Mathematics, University of Science and Technology China, Hefei, Anhui Pronvince 230026, P.R. China}
\email{craigvan@ustc.edu.cn}
\date{December 25, 2010}
\keywords{K\"{a}hler-Einstein, pseudoconvex, CR structure, Sasakian}
\subjclass{Primary 32Q20, Secondary 32T15 }

\begin{document}

\begin{abstract}
Extending the results of S. Y. Cheng and S.-T. Yau it is shown that a strictly pseudoconvex domain $M\subset X$ in a complex manifold
carries a complete K\"{a}hler-Einstein metric if and only if its canonical bundle is positive, i.e. admits an Hermitian connection with
positive curvature.  We consider the restricted case in which the CR structure on $\partial M$ is normal.
In this case $M$ must be a domain in a resolution of the Sasaki cone over $\partial M$.  We give a condition on a
normal CR manifold which it cannot satisfy if it is a CR infinity of a K\"{a}hler-Einstein manifold.  We are able to mostly
determine those normal CR 3-manifolds which can be CR infinities.

We give many examples of K\"{a}hler-Einstein strictly pseudoconvex manifolds on bundles and resolutions.
In particular, the tubular neighborhood of the zero section of every negative holomorphic vector bundle on a compact complex manifold
whose total space satisfies $c_1 <0$ admits a complete K\"{a}hler-Einstein metric.
\end{abstract}

\maketitle

\tableofcontents

\section{Introduction}

S. Y. Cheng and S.-T. Yau proved in~\cite{ChenYau} that a bounded strictly pseudoconvex domain in $\C^n$ admits a complete negative scalar curvature K\"{a}hler-Einstein metric.  Their arguments also extended to other types of domains, such as a pseudoconvex domain which
is the intersection of pseudoconvex domains with $C^2$ boundary.
Many cases of domains in arbitrary complex manifolds are already dealt with in~\cite{ChenYau}, and in~\cite{MokYau}.
In~\cite{MokYau} N. Mok and S.-T. Yau proved the existence of a K\"{a}hler-Einstein metric on strictly pseudoconvex domains under some hypotheses.  These include, in particular, domains in Stein manifolds and domains which admit a negative Ricci curvature metric.
This article considers the existence of a complete negative scalar curvature K\"{a}hler-Einstein metric on a strictly pseudoconvex domain of an arbitrary complex manifold.

\begin{thm}\label{thm:main}
Suppose $M$ is a strictly pseudoconvex domain in $M'$.  Then $M$ admits an unique complete K\"{a}hler-Einstein metric of negative scalar
curvature if and only if $\mathbf{K}_M$ is positive, i.e. admits an Hermitian metric with positive curvature.
\end{thm}

Strict pseudoconvexity means that the boundary $S:=\partial M$ has a positive CR structure.
We consider the case in which the CR structure on $S$ is in addition normal, that is, admits transverse
vector field preserving it.  We prove that any strictly pseudoconvex manifold with a normal CR structure on $S$
must be a domain in a resolution of the Sasaki cone $C(S)$ of the natural Sasaki structure on $S$.
This severely restricts the strictly pseudoconvex manifolds with a normal CR structure on the boundary.
In particular, a domain in a Stein manifold must be in $\C^n$ and have a boundary diffeomorphic to $\mathbb{S}^{2n-1}$
with CR structure a deformation of the standard one.  We consider the following:
\begin{prob}\label{prob:bound-K-E}
Which positive normal CR manifold $(S,D,J)$ is the conformal boundary of K\"{a}hler-Einstein manifold?
\end{prob}
There are far too many CR structures on even simple manifolds for this to be a tractable problem in general.
This is true even for $S=\mathbb{S}^{2n-1}$.  It is shown in~\cite{BoyGalKol,BoyGalKolThom} that there are inequivalent families
of CR structures on all odd dimensional spheres with the number of deformation classes growing doubly exponentially with the
dimension.  In particular, $\mathbb{S}^5$ has 68 inequivalent deformation classes.  More interesting is that these CR structures
have associated Sasaki-Einstein metrics.  But in the present context, we prove here that they cannot be the conformal boundary
of a K\"{a}hler-Einstein manifold.  We prove that any simply connected normal CR manifold $(S,D,J)$ satisfying the topological condition
for a compatible Sasaki-Einstein metric, $c_1(D)=0$ and $c_1^B >0$, are negative examples of \ref{prob:bound-K-E}.
Thus not surprisingly, those $(S,D,J)$ which can be the boundary at infinity of a complete Ricci-flat manifold as considered in~\cite{vanCo1} are excluded.  We are able to mostly answer Problem~\ref{prob:bound-K-E} in dimension 3, just leaving open some cases of finite quotients of $\mathbb{S}^3$.

We give many examples in which Theorem~\ref{thm:main} applies.  This gives many examples with normal CR structures at infinity
and otherwise.  An easy case is that of negative holomorphic bundles over a compact complex manifold.
\begin{cor}\label{cor:main}
Let $\pi:\mathbf{E}\rightarrow N$ be a negative holomorphic bundle over a compact complex manifold $N$.
If $-c_1(M)-c_1(\mathbf{E})>0$ on $N$, then the disk subbundles $D\subset \mathbf{E}$ admit unique complete K\"{a}hler-Einstein metrics of negative scalar curvature.
\end{cor}
We also construct some examples on resolutions of hypersurface singularities and on some familiar resolutions of quotient singularities.

\section{Background}

Let $S$ be a real $2n-1$-dimensional manifold.  A \emph{CR structure} on $S$ is a pair $(D,J)$ consisting of a distribution
$D\subset TS$ of real $2n-2$-dimensional hyperplanes and an almost complex structure $J$ on $D$ such that, if
$D^{1,0} \subset D\otimes\C \subset TS\otimes\C$ denotes the type $(1,0)$-vectors, the \emph{formal integrability} condition holds:
\begin{equation}\label{eq:formal-int}
[D^{1,0} ,D^{1,0} ]\subseteq D^{1,0}.
\end{equation}
The \emph{Levi form} $\mathbf{L}^D :D\times D\rightarrow TS/D $ is defined by $\mathbf{L}^D (X,Y)=-[X,Y]\ \mod D$,
for $X,Y \in\Gamma(D)$.

It is easy to check that (\ref{eq:formal-int}) is equivalent to both $[X,JY] +[JX,Y] \in\Gamma(D)$ and the vanishing of
the Nijenhuis tensor
\begin{equation}\label{eq:Nijen}
\mathcal{N}_J(X,Y)=[JX,JY]-[X,Y]-J\bigl([X,JY]-[JX,Y]\bigr)=0,\quad X,Y\in\Gamma(D).
\end{equation}
Note that the former condition implies $\mathbf{L}^D$ is $J$-invariant, i.e. $\mathbf{L}^D(J\cdot,J\cdot) =\mathbf{L}^D(\cdot,\cdot)$.

We will always assume $S$ is orientable, so $TS/D$ is a trivial real line bundle.  Then there is a 1-form $\eta$ with
$D =\ker\eta$, and we may identify $\mathbf{L}^D =d\eta|_D$.  Note that $\mathbf{L}^D$ is only defined up to a choice of orientation
of $TS/D$ and a positive conformal factor.
\begin{defn}
The CR structure $(D,J)$ is \emph{strictly pseudoconvex} if, for a given orientation of $TS/D$, the Levi form $\mathbf{L}^D$ is positive,
i.e. $\mathbf{L}^D (X,JX)>0$ for any non-zero $X\in D$.
\end{defn}

Note that formal integrability does not in general imply integrability, that is that $(S,D,J)$ is a real hypersurface in a complex manifold.
The analogue of the Newlander-Nirenberg theorem only holds for analytic CR manifolds~\cite{Niren1,Niren2}.

\begin{defn}
A \emph{finite manifold} is a pair $(M,M')$ of complex manifolds with $M$ an open relatively compact submanifold of $M'$ with smooth
non-empty boundary $\partial M$.
\end{defn}

Let $\phi$ be a defining function of $M\subset M'$.  That is, $\phi$ is $C^\infty$ in a neighborhood of $\ol{M}$, $M=\{\phi <0\}$,
and $d\phi\neq 0$ on $\partial M$.

Let $J$ denote the complex structure of $M'$.
The real $2n-1$-dimensional manifold $S:=\partial M$ has the CR structure $(D,J)$ where $D:= TS \cap J TS$ and $J$ is restricted to $D$.

Define a 1-form on $S$
\begin{equation}
\eta:= d^c \phi |_S ,
\end{equation}
where $d^c := \frac{\sqrt{-1}}{2}(\ol{\partial} -\partial)$.  Then it is easy to see that $D=\ker\eta$, and the Levi form is
$\mathbf{L}^D =d\eta|_D$, which is a $(1,1)$-form on $D$ as follows from the comments after (\ref{eq:Nijen}).
\begin{defn}
We say that the pair $(M,M')$ is \emph{strictly pseudoconvex} or that $M$ is a strictly pseudoconvex domain, if the induced CR-structure
$(D,J)$ is.
\end{defn}

One can check that by altering $\phi$, for instance considering $e^{A\phi}\phi$ for a constant $A>0$, one may assume that
$\phi$ is strictly plurisubharmonic on a neighborhood of $S=\partial M$.  That is, $\sqrt{-1}\partial\ol{\partial}\phi$ is
a positive $(1,1)$-form.

A strictly pseudoconvex domain $M\subset M'$ is a particular type of 1-convex manifold.  A complex manifold $X$ is said to
be 1-convex if there is a Stein space $Y$, a proper holomorphic surjective mapping $\pi:X\rightarrow Y$ satisfying
$\pi_* \mathcal{O}_X =\mathcal{O}_Y$, and a finite set $A\subset Y$ such that if $E=\pi^{-1}(A)$ the map
$\pi: X\setminus E\rightarrow Y\setminus A$ is a biholomorphism.  Then $Y$ is called the Remmert reduction of $X$ and
$E$ is called the exceptional set of $X$.  Note that $E$ is the maximal compact analytic subvariety of $X$, i.e. the union
of all compact analytic subvarieties of dimension $\geq 1$.

Let $\xi\in\Gamma(TS)$ be a vector field on $S$ so that
\begin{equation}\label{eq:dir-sum}
TS=D\oplus\R\xi.
\end{equation}
We extend $J$ to a $(1,1)$ tensor $\Phi$ on $S$ by
\begin{equation}
\Phi(X) =JX,\ \text{for }X\in D,\text{ and }\Phi\xi =0.
\end{equation}
\begin{defn}
The CR structure $(D,J)$ on $S$ is \emph{normal} if there is a $\xi$ satisfying (\ref{eq:dir-sum}) whose flow preserves $(D,J)$.
Or in other words $\mathcal{L}_\xi \Phi=0$.
\end{defn}
Suppose $\eta(\xi)>0$ for an oriented contact form $\eta$.  So we may assume by changing by a conformal factor that
$\eta$ is the unique 1-form with $\ker\eta=D$ and $\eta(\xi)=1$, i.e. $\xi$ is the Reeb vector field of $\eta$.
Then $(D,J)$ is normal if and only if
\begin{equation}\label{eq:normal}
\mathcal{N}_\Phi =\xi\otimes d\eta.
\end{equation}
In fact, it is easy to check that (\ref{eq:normal}) is equivalent to (\ref{eq:Nijen}) and $\mathcal{L}_\xi \Phi=0$.
And these two conditions are equivalent to the integrability of the almost complex structure on the \emph{cone}
$C(S):=\R_+ \times S$
\begin{equation}\label{eq:cone-cst}
I(X) =\Phi(X) -\eta(X)r\partial_r,\quad I(r\partial_r) =\xi,
\end{equation}
where $r$ is the radial coordinate on $\R_+$ and $X\in TS$.

If the CR structure $(D,J)$ is positive and normal, then with $\xi$ positively oriented as above, we have a natural metric
\begin{equation}
g(X,Y):=\frac{1}{2}d\eta(X,\Phi Y) +\eta(X)\eta(Y),\quad X,Y\in TS.
\end{equation}
In this case $S$ has a special type of metric contact structure, known as a \emph{Sasaki structure}, which we denote by
$(g,\xi,\eta,\Phi)$.  See~\cite{BoyGal} for more details.

We denote by $\mathfrak{CR}(S,D,J)$ the automorphism group of the CR manifold $(S,D,J)$ and its Lie algebra by $\mathfrak{cr}(S,D,J)$.
R. Schoen~\cite{Schoen} proved the following result.
\begin{thm}
The CR automorphism group of a strictly pseudoconvex CR manifold $(M,D,J)$ is proper unless $M$ is either $\mathbb{S}^{2n-1}$ or the Heisenberg group $\mathcal{H}_{2n-1}$ with the standard CR structures.
\end{thm}
\begin{cor}
If $(M,D,J)$ is a compact strictly pseudoconvex CR manifold, then the CR automorphism group $\mathfrak{CR}(S,D,J)$ is compact unless $M=\mathbb{S}^{2n-1}$ with the standard CR structure, in which case the CR automorphism group is $\operatorname{PSU}(1,n)$.
\end{cor}

It will be useful to consider the space of compatible Sasaki structures on a normal strictly pseudoconvex CR manifold $(S,D,J)$.
See~\cite{BoyGalSim1,BoyGalSim2} for more details.
\begin{defn}
Let $(S,D,J)$ be a strictly pseudoconvex CR structure.  A vector field $X\in\mathfrak{cr}(S,D,J)$ is \emph{positive} if
$\eta(X)>0$ for an oriented contact form $\eta$.  We denote by $\mathfrak{cr}^+(S,D,J)$ the space of all positive elements of $\mathfrak{cr}(S,D,J)$.
\end{defn}
It is not difficult to see that $\mathfrak{cr}^+(S,D,J)$ is isomorphic the space of Sasaki structures compatible with $(D,J)$,
$\mathfrak{cr}^+(S,D,J)$ is an open convex cone in $\mathfrak{cr}(S,D,J)$, and is invariant under the adjoint action of
$\mathfrak{CR}(S,D,J)$.
\begin{defn}
Let $(S,D,J)$ be a normal strictly pseudoconvex CR manifold.  The \emph{Sasaki cone} $\kappa(S,D,J)$ is the moduli space of Sasaki structures
compatible with $(D,J)$ on $S$.  We have
\begin{equation}
\kappa(S,D,J) =\mathfrak{cr}^+(S,D,J)/\mathfrak{CR}(S,D,J).
\end{equation}
\end{defn}

Choose a maximal torus $T_k \subset\mathfrak{CR}(S,D,J)$ of rank $k$, $1\leq k\leq n$, with Lie algebra $\mathfrak{t}_k$.  Then if
$\mathfrak{t}_k^+$ denotes the subspace of positive elements, we have
\begin{equation}\label{eq:Sasaki-cone}
\kappa(S,D,J) =\mathfrak{t}_k^+/\mathcal{W},
\end{equation}
where $\mathcal{W}$ is the Weyl group of $T_k \subset K\subseteq\mathfrak{CR}(S,D,J)$ for a maximal compact subgroup $K$.  Of course,
$K=\mathfrak{CR}(D,J)$ unless $\mathfrak{CR}(D,J)=\operatorname{PSU}(1,n)$.  Let $Z_k \subset\mathfrak{t}_k$ be the lattice of \emph{integral} elements, that is $Z_k =\{\xi\in\mathfrak{t}_k : \exp(2\pi\xi) =1\}$.  And define $Z^+_k =Z_k \cap\mathfrak{t}^+_k$.  Then
every $\xi\in Z^+_k$ defines
a quasi-regular Sasaki structure, that is all the orbits of the Reeb vector field $\xi$ close to give a locally free
$\operatorname{U}(1)$-action.  The $\operatorname{U}(1)$-action on $S$ extends to a locally free holomorphic $\C^*$-action on $C(S)$.
And $C(S)$ is biholomorphic to the total space minus the zero section $\mathbf{L}^\times$ of a negative holomorphic orbibundle over a
K\"{a}hler orbifold $W$ (cf.~\cite{BoyGal}).

\section{The K\"{a}hler-Einstein metric}

\subsection{The approximate metric}

Let $M\subset M'$ be a smooth strictly pseudoconvex domain in a K\"{a}hler manifold $(M',g_0)$.
And let $\phi$ be a plurisubharmonic defining function which is strictly plurisubharmonic on a neighborhood of
$\partial M$.  Then $h=-\log(-\phi)$ is strictly plurisubharmonic near $\partial M$, and $dd^c h$ is
the K\"{a}hler form of a metric near $\partial M$ which in coordinates is
\begin{equation}
h_{i\ol{\jmath}} =\frac{\phi_{i\ol{\jmath}}}{-\phi} +\frac{\phi_i \phi_{\ol{\jmath}}}{\phi^2}.
\end{equation}
Computation gives
\begin{equation}
h^{i\ol{\jmath}} =(-\phi)\left(\phi^{i\ol{\jmath}} +\frac{\phi^i \phi^{\ol{\jmath}}}{\phi-|d\phi|^2}  \right),
\end{equation}
where $\phi^{i\ol{\jmath}} =(\phi_{i\ol{\jmath}})^{-1}$, $\phi^i =\sum\phi^{i\ol{\jmath}}\phi_{\ol{\jmath}},$ and
$|d\phi|^2 =\phi^{i\ol{\jmath}}\phi_i \phi_{\ol{\jmath}}$.

It is also easy to see that
\begin{equation}
h^{i\ol{\jmath}} h_i h_{\ol{\jmath}} =\frac{|d\phi|^2}{|d\phi|^2 -\phi} \leq 1.
\end{equation}
Thus since $h(x)\rightarrow\infty$ as $x\rightarrow\partial M$, the metric $h_{i\ol{\jmath}}$ is complete toward $\partial M$.
Therefore, \emph{a fortiori} the metric $g_{i\ol{\jmath}} =(g_0)_{i\ol{\jmath}} +h_{i\ol{\jmath}}$ with K\"{a}hler form
\begin{equation}
\omega =\omega_0 +dd^c h,
\end{equation}
is a complete K\"{a}hler metric on $M$.

\subsection{Existence of the metric}

We will consider the existence of a complete K\"{a}hler-Einstein metric on $M$, that is a K\"{a}hler metric $g$ with
\begin{equation}\label{eq:K-E}
\Ric_g =-\lambda g,\quad \lambda>0.
\end{equation}
For convenience we will set $\lambda=n+1$.  If $g$ is a complete K\"{a}hler metric on $M$ with K\"{a}hler form $\omega$,
suppose we have $F\in C^\infty$ with
\begin{equation}\label{eq:dd-bar-ric}
(n+1)\omega +\ric(\omega) =dd^c F.
\end{equation}
Then a solution to the Monge-Amp\`ere equation
\begin{equation}\label{eq:M-A}
(\omega +dd^c u)^n =e^{(n+1)u +F} \omega^n
\end{equation}
provides a K\"{a}hler metric $\omega'=\omega +dd^c u$ satisfying (\ref{eq:K-E}).  Equation (\ref{eq:M-A}) on noncompact
manifolds was extensively studied by S.-Y Cheng and S.-T. Yau~\cite{ChenYau}.  See also~\cite{TianYau}.
There it was proved that (\ref{eq:M-A}) has a unique solution if $F\in C^{3,\alpha}(M)$ and $(M,g)$ has bounded geometry.

We use this method to find a complete solution to (\ref{eq:K-E}) where $M\subset M'$ is a strictly pseudoconvex domain in a
K\"{a}hler manifold $M'$.  With $\phi$ a defining function of $M$ and $\omega_0$ a K\"{a}hler form on $M'$ we consider
the complete metric with K\"{a}hler form
\begin{equation}
\omega =\omega_0 -dd^c \log(-\phi).
\end{equation}

If a line bundle $\mathbf{L}$ is given by a system of charts and transition functions $(U_{\alpha}, g_{\alpha\beta})$,
then an Hermitian metric on $\mathbf{L}$ is given by a system $\{h_\alpha\}$ of smooth positive functions on $\{U_\alpha\}$
which satisfy $h_\alpha =|g_{\beta\alpha}|^2 h_\beta$ on $U_{\alpha}\cap U_{\beta}$.   In particular, we will use
that any other Hermitian metric $h'$ on $\mathbf{L}$ is of the form $h'=e^f h$ for $f\in C^\infty$.

An holomorphic line bundle $\mathbf{L}$ is \emph{positive} if it has an Hermitian metric $h$ such that the
curvature of the associated Chern connection, $\Theta_{\mathbf{L}} =-\partial\ol{\partial}\log h$, satisfies
$\frac{\sqrt{-1}}{2\pi}\Theta_{\mathbf{L}} >0$, i.e. is a positive $(1,1)$-form.

The following theorem is mostly due to S. Y. Cheng and S. T. Yau~\cite{ChenYau}.
\begin{thm}\label{thm:K-E}
Let $(M,M')$ be a strictly pseudoconvex finite manifold.   Then $M$ admits a complete K\"{a}hler-Einstein metric of negative scalar
curvature if and only if $\mathbf{K}_{M}$ is positive.
\end{thm}
\begin{proof}
Let $h$ be a positive Hermitian metric on $\mathbf{K}_{M}$ and let $h'$ be any connection on $\mathbf{K}_{M'}$.
Choose $\varsigma\in C^\infty(\R)$ with $\varsigma(x)=1$ for $x\geq 1$ and $\varsigma(x)=0$ for $x\leq 1/2$.
Set $\varsigma_R (x):=\varsigma(\frac{x}{R})$.  Consider the metric
$\tilde{h}=\varsigma_R(-\phi)h +(1-\varsigma_R(-\phi))h'$, which has positive curvature on $\{-\phi >R\}\subset M$.
Choose $R>0$ sufficiently small that this set contains the maximal compact analytic subset $E$ of $M$, and
choose a plurisubharmonic function $\psi$ on $M'$ which is strictly plurisubharmonic away from $E$.
Then $e^{-A\psi}\tilde{h}$ has positive curvature on a neighborhood of $M\subset M'$ for $A\gg 0$.

Suppose $h$ is a metric on $\mathbf{K}_{M'}$ with $\omega_0 =\frac{\sqrt{-1}}{(n+1)}\Theta_h$ positive on a neighborhood of
$M$.  Then the volume form $\frac{1}{n!}\omega^n_0$ defines an Hermitian metric on $\mathbf{K}_{M'}$ by
\begin{equation}\label{eq:can-metric}
\|\Omega\|^2 := \left(\frac{i}{2}\right)^n (-1)^{\frac{n(n-1)}{2}} \frac{\Omega\wedge\ol{\Omega}}{\omega_0^n},
\end{equation}
for $(n,0)$-form $\Omega$.  Then $\|\cdot\|^2 =e^f h$, for some $f\in C^\infty(M')$.  So we have
\begin{equation}\label{eq:omega0-Ricci}
(n+1)\omega_0 +\ric(\omega_0) =dd^c f.
\end{equation}
We define
\begin{equation}\label{eq:F}
F=\log\left[\frac{e^f (-\phi)^{-(n+1)}\omega_0^n}{(\omega_0 -dd^c \log(-\phi))^n}  \right].
\end{equation}
Then for the metric $\omega=\omega_0 -dd^c \log(-\phi)$ we have that $F$ satisfies (\ref{eq:dd-bar-ric}).
It is easy to see that $F\in C^\infty(\overline{M})$.  In fact one checks that
\begin{equation}
\frac{e^f (-\phi)^{-(n+1)}\omega_0^n}{(\omega_0 -dd^c \log(-\phi))^n}\Big | _{\partial M} =\frac{e^f \omega_0^n}{|d\phi|^2 (dd^c \phi)^n}\Big|_{\partial M}.
\end{equation}
Then the proof in~\cite{ChenYau} shows that (\ref{eq:M-A}) has a unique solution $u\in C^\infty(M)$.
The proof follows from an application of the generalized maximum principle to formulae of~\cite{Yau} to obtain the
necessary \emph{a priori} estimates.

The converse is clear.  Since if $g_0$ is K\"{a}hler-Einstein, then the curvature of (\ref{eq:can-metric}) satisfies
$\frac{\sqrt{-1}}{2\pi}\Theta =-\frac{1}{2\pi}\ric(g_0)>0$.
\end{proof}

\subsection{uniqueness}

The following uniqueness result is due to S.-T. Yau and follows from a more general Schwartz lemma.
\begin{prop}[\cite{MokYau}]\label{prop:unique}
Let $(M_1 ,g_1)$ and $(M_2,g_2)$ be complete complete K\"{a}hler-Einstein manifolds of negative scalar curvature, normalized
to have equal Einstein constants, $\lambda$.  If $\sigma :M_1 \rightarrow M_2$ is a biholomorphism, then $\sigma^* g_2 =g_1$.
\end{prop}
\begin{proof}
Let $\mu_{M_i},\ i=1,2$ be the respective volume forms, and define $f=\sigma^* \mu_{M_2}/\mu_{M_1}$.
Then we have
\begin{equation}\label{eq:vol}
\begin{split}
\Delta\log f & = n\lambda - g_1^{i\ol{\jmath}}\sigma^* \ric(g_2)_{i\ol{\jmath}} \\
             & = n\lambda -\lambda g_1^{i\ol{\jmath}}(\sigma^* g_2)_{i\ol{\jmath}} \\
\end{split}
\end{equation}
The arithmetic-geometric inequality applied to the second term on the right of (\ref{eq:vol}) gives
\begin{equation}
\Delta\log f\geq n\lambda -\lambda n f^{1/n}.
\end{equation}
From which we have
\begin{equation}
\Delta f \geq n\lambda f -n\lambda f^{\frac{n+1}{n}},
\end{equation}
and it follows from the maximum principle as in~\cite[{Lemma 1.1}]{TianYau} that $\sup f \leq 1$.
Applying the same argument to $\sigma^{-1}$ gives $\sigma^* \mu_{M_2} =\mu_{M_1}$, from which we have
$\sigma^* \ric(g_2) =\ric(g_1)$ and $\sigma^* g_2 =g_1$.
\end{proof}

Let $\mathfrak{Hol}(M)$ denote the group of biholomorophisms of $M$, $\mathfrak{Isom}(M,g)$ the group of isometries of $(M,g)$,
and $\mathfrak{hol}(M)$, $\mathfrak{isom}(M,g)$ their respective Lie algebras.
We also have the following easy converse to Proposition~\ref{prop:unique}.
\begin{prop}
Let $\sigma :M \rightarrow M$ be an isometry of K\"{a}hler-Einstein strictly pseudoconvex manifold.
Then $\sigma$ is a biholomorphism up to conjugation, i.e. $\sigma^* J =\pm J$ where $J$ is the complex structure of
$M$.  Thus
\begin{equation}\label{eq:isom-hol}
\mathfrak{isom}(M,g)=\mathfrak{hol}(M).
\end{equation}

Furthermore, if $X\in\mathfrak{isom}(M,g)$, then $JX\notin\mathfrak{isom}(M,g)$.
\end{prop}
\begin{proof}
First note that since $(M,g)$ has curvature asymptotic to constant $-2$ holomorphic bisectional curvature it must be
irreducible as a K\"{a}hler manifold.  There are two, $J$ and $\sigma^* J$, parallel complex structures on $M$.
Since $(M,J)$ is irreducible, either $\sigma^* J =\pm J$,
or the holonomy group $\Hol(g_1)\subseteq Sp(\frac{n}{2})$.  In other words, in the second case one can show the existence
of three parallel complex structures $J_1, J_2, J_3$ satisfying the quaternionic identities.  But in this case $\Ric_{g} =0$,
and (\ref{eq:isom-hol}) follows.

Suppose that $X,JX\in\mathfrak{isom}(M,g,J)$.  The following argument is due to S. Kobayashi \cite[Ch. III,\S 1]{Kobay}.
Define
\[ A_X =\mathcal{L}_X -\nabla_X.\]
Since $X\in\mathfrak{hol}(M)$ and $(M,g)$ is K\"{a}hler, we have
\begin{equation}\label{eq:A}
J A_X =A_X J =A_{JX}.
\end{equation}
We have
\begin{equation}
g(A_{JX} Y,Z)+ g(Y,A_{JX} Z)=0
\end{equation}
and from (\ref{eq:A})
\begin{equation}
g(A_{X}JY,Z)- g(JY,A_{X} Z)=0.
\end{equation}
It follows that $A_X$ is symmetric.  But since it is also skew-symmetric, we have $A_X =-\nabla X =0$.
This implies $\Ric_g (X,X)=0$, a contradiction.
\end{proof}

\subsection{Boundary behavior}

We consider the boundary behavior of the metric $g_{i\ol{\jmath}}= (g_0)_{i\ol{\jmath}} + h_{i\ol{\jmath}}$ and the Einstein metric $g'_{i\ol{\jmath}}$ of Theorem~\ref{thm:K-E}, where $h=-\log(-\phi)$ for a defining function $\phi$ and $g_0$ is a K\"{a}hler metric on $\ol{M}$.  First a straight forward calculation as in~\cite{ChenYau} gives the Christoffel symbols $\Gamma^k_{ij}$ and the curvature $R_{i\ol{\jmath}l\ol{k}}$ of $h_{i\ol{\jmath}}$ near the boundary of $M$.

\begin{gather}\label{eq:Christ}
\Gamma^k_{ij} =\phi_{ij\ol{l}} \phi^{k\ol{l}} +\frac{\phi_i \delta_j^k +\phi_j \delta_i^k}{-\phi}+\frac{1}{\phi-|d\phi|^2}\bigl(\phi_{ij\ol{l}}\phi^{\ol{l}}\phi^k -\phi_{ij}\phi^k \bigr)\\\label{eq:curv}
R_{i\ol{\jmath}l\ol{k}} =-(g_{i\ol{\jmath}}g_{l\ol{k}}+g_{i\ol{k}}g_{l\ol{\jmath}})+\frac{1}{\phi}\bigl(R^{\phi}_{i\ol{\jmath} l\ol{k}}+\frac{1}{\phi -|d\phi|^2}(\phi_{,il} \phi_{,\ol{\jmath}\ol{k}})\bigr)
\end{gather}
Here $R^\phi$ denotes the curvature and $\phi_{,ij}$ the covariant derivative with respect to $\phi_{i\ol{\jmath}}$.

The optimal regularity and asymptotic behavior of the solution $u\in C^\infty(M)$ to (\ref{eq:M-A}) of Theorem~\ref{thm:K-E} was
given in~\cite{LeeMel} for $M\subset\C^n$.  The proof works with only minor modifications to an arbitrary strictly pseudoconvex $M\subset M'$ with the initial metric $g_{i\ol{\jmath}}$.  An essential step is to find a defining function $\phi_0$ so that $F$ defined
in (\ref{eq:F}) vanishes to high order on $\partial M$.  The following was first proved by C. Fefferman~\cite{Feff2} for $M\subset\C^n$.
\begin{lem}
There exists a defining function $\phi_0$ of $M\subset M'$ so that $F$ given in (\ref{eq:F}) satisfies
\begin{equation}
F =O(\phi_0^{n+1})
\end{equation}
\end{lem}
\begin{proof}
We seek $\beta\in C^\infty(\ol{M})$ so that $\phi' =e^\beta \phi$ so that
\begin{equation}\label{eq:F-0}
\frac{e^f (-\phi')^{-(n+1)}\omega_0^n}{(\omega_0 -dd^c \log(-\phi'))^n} \rightarrow 1\text{ on }\partial M.
\end{equation}
But since
\begin{equation}
\frac{(\omega_0 -dd^c \log(-\phi))^n}{(\omega_0 -dd^c \log(-\phi'))^n}=\frac{(\omega_0 -dd^c \log(-\phi))^n}{(\omega_0 -dd^c \log(-\phi)-dd^c \beta)^n} \rightarrow 1\text{ on }\partial M,
\end{equation}
we may take $\beta =\frac{F}{n+1}$, and (\ref{eq:F-0}) is satisfied.  Then the inductive argument in the proof of~\cite{LeeMel}
goes through with the operator
\begin{equation}
\beta \rightarrow \frac{(\omega_0 -dd^c \log(-\phi)+dd^c \beta)^n e^{-f} (-e^{-\beta}\phi)^{n+1}}{\omega_0^n}
\end{equation}
substituting that used there.
\end{proof}

The results of~\cite{LeeMel} on the asymptotic behavior of the solution $u\in C^\infty(M)$ to (\ref{eq:M-A}) with defining function $\phi_0$ are valid in this situation.  One can define H\"{o}lder spaces $C^{k,\alpha}(M)$ with respect to the metric
$(g_0)_{i\ol{\jmath}} + h_{i\ol{\jmath}}$.  Then if $F$ given in (\ref{eq:F}) vanishes to order $0<r< n+1$, we have
\begin{equation}\label{eq:u-asym-hol}
u\in\bigcap_k \phi^r C^{k,\alpha}(M).
\end{equation}

Moreover, there is an asymptotic expansion of $u$.  There are
$\alpha_j \in C^\infty(\ol{M}),\ j\geq 1$, such that for $N\in\N$
\begin{equation}\label{eq:u-asym}
u-\sum_{j=1}^N \alpha_j \phi_0^{(n+1)j} (\log(-\phi_0))^j \in C^{(n+1)(N+1)-1,\alpha}(\ol{M}),
\end{equation}
and vanishes to order $(n+1)(N+1)-1$.

\begin{prop}
Let $g_{i\ol{\jmath}}$ be either the metric $(g_0)_{i\ol{\jmath}} + h_{i\ol{\jmath}}$ or $(g_0)_{i\ol{\jmath}} + h_{i\ol{\jmath}} +u_{i\ol{\jmath}}$ solving (\ref{eq:K-E}), i.e. the K\"{a}hler-Einstein metric, on a strictly pseudoconvex $M$,
and let $r=\dist(o,x)$ be the distance from a fixed point $o\in M$, then the curvature of $g_{i\ol{\jmath}}$ satisfies
\begin{equation}
R_{i\ol{\jmath}l\ol{k}} =-(g_{i\ol{\jmath}}g_{l\ol{k}}+g_{i\ol{k}}g_{l\ol{\jmath}}) + O(e^{-2r}).
\end{equation}
Thus metric $g$ is asymptotically of constant holomorphic sectional curvature $-2$ and are \emph{asymptotically complex hyperbolic}(ACH).

If $g$ is the K\"{a}hler-Einstein metric on a strictly pseudoconvex $M\subset M$ and $\phi$ is any defining function, then
$(-\phi)g \rightarrow \mathbf{L}^D (\cdot,J\cdot)$ on $D\subset T\partial M$, where the Levi form $\mathbf{L}^D$ is of course only defined up to a conformal factor.
\end{prop}

\subsection{Comments on the theorem and the $\partial\ol{\partial}$-lemma}\label{sec:com-dd-bar}

One could also consider the weaker condition that $-c_1 (\tilde{M})$ is represented by a positive $(1,1)$-form.
This is \textit{a priori} weaker assumption as the $\partial\ol{\partial}$-lemma does not generally hold on a 1-convex manifold.
It remains whether this weaker assumption is a sufficient condition for Theorem~\ref{thm:K-E}.  The following is
easy.
\begin{lem}
Let $X$ be a complex manifold.  Then $X$ satisfies the $\partial\ol{\partial}$-lemma if and only if for any
holomorphic line bundle $\mathbf{L}$ any $\omega\in c_1(\mathbf{L})$ is represented by $\frac{\sqrt{-1}}{2\pi}\Theta_h$
for some Hermitian metric $h$.
\end{lem}

We make the following
\begin{conj}\label{conj:weak-hyp}
A strictly pseudoconvex domain $M$ admits a complete K\"{a}hler-Einstein metric if and only if there is a K\"{a}hler form
$\omega\in -c_1(M)$.
\end{conj}
We will give some results that make the conjecture plausible.  At least the following results will show that
constructing a counterexample to Conjecture~\ref{conj:weak-hyp} would be very difficult.
We consider some examples of 1-convex surfaces due to M. Col\c{t}oiu~\cite{Colt} on which the $\partial\ol{\partial}$-lemma does not hold, but nevertheless Theorem~\ref{thm:K-E} applies.

Let $X$ be an 1-convex manifold with exceptional set $E$.  We denote by $B(X)\subset\pic X$ the subgroup of line bundles
$\mathbf{L}$ which are topologically trivial on $X$ and holomorphically trivial in a neighborhood of $E$.
\begin{prop}[\cite{Colt}]\label{prop:B-isom}
For a 1-convex manifold with exceptional set $E$ there is a group isomorphism
\[ B(X) \overset{\sim}{\rightarrow}  H^1(E,\Z)/\im\bigl[H^1(X,\Z)\rightarrow H^1(E,\Z) \bigr].\]
\end{prop}

Let $C_1$ and $C_2$ be smooth curves in $\cps^2$ intersecting transversely of degrees $d_1 \geq 3$ and $d_2 >d_1$ respectively.
Let $\pi: Y\rightarrow\cps^2$ be the blow up at each of the $d_1 d_2$ points $p_1,\ldots,p_{d_1 d_2}$ of intersection $C_1 \cdot C_2$.
If $\hat{C}_i ,i=1,2$ are their strict transforms, then as divisors
\begin{equation}
\hat{C}_i =\pi^* d_i H-\sum_{j=1}^{d_1 d_2} E_j ,\quad i=1,2,
\end{equation}
where $E_j \ ,j=1,\ldots,d_1 d_2$ are the exceptional divisors.  Then $\hat{C}_1^2 =d_1^2 -d_1 d_2 <0$ and
$\hat{C}_2^2 =d_2^2 -d_1 d_2 >0$.  The first inequality, by a theorem of Grauert, implies that $\hat{C}_1$ is exceptional,
and the second inequality implies that $X=Y\setminus\hat{C}_2$ is 1-convex.  And it turns out that $\hat{C}_1$ is the entire
exceptional set.  See~\cite{Colt} for details.

In addition it is shown that $\pi_1(X)=1$ and by the genus formula we have $g(\hat{C}_1) =\frac{(d_1 -1)(d_1 -2)}{2}$.
Thus
\begin{equation}
B(X) =\Z^{2g}.
\end{equation}

Since $K_Y =\pi^* (-3H) +\sum_{j=1}^{d_1 d_2} E_j$, we have
\begin{equation}
D=K_Y +2\hat{C}_2 =\pi^* (2d_2 -3)H -\sum_{j=1}^{d_1 d_2} E_j.
\end{equation}
One can show using the Nakai-Moishezon criterion that $D >0$.
Clearly, $D^2>0$.  We need to show that a curve $C\subset\cps^2$ with $\deg C =d$ does not intersect
$d(2d_2 -3)$ of the points $p_1,\ldots, p_{d_1 d_2}$ when counted with the multiplicity of $C$ at each point.
But by B\'ezout's theorem
\begin{equation}
\sum_j i(C,C_1, q_j) =dd_1,
\end{equation}
where the sum is over the points of intersection of $C$ with $C_1$.  We have
\begin{equation}
i(C,C_1,q_j)\geq \mu_{q_j}(C)\mu_{q_j}(C_1) =\mu_{q_j}(C),
\end{equation}
where $\mu_{q_j}(C),\ \mu_{q_j}(C_1)$ denote the multiplicity of $C$, respectively $C_1$, at $q_j$.
And
\begin{equation}
\sum_{j} \mu_{q_j}(C)\leq dd_1 < d(2d_2 -3)
\end{equation}
shows that $C$ cannot intersect $d(2d_2 -3)$ of the points $p_1,\ldots, p_{d_1 d_2}$ when counted with the multiplicity.

For any relatively compact strictly pseudoconvex domain $M\subset X$ we have $\mathbf{K}_{M} >0$.

Let $\mathcal{E}\subset\mathcal{A}^{1,1}(X)$ be the space of smooth exact $(1,1)$-forms.
The $\partial\ol{\partial}$-lemma hold on a manifold $X$ precisely when the map
\begin{equation}\label{eq:dd-bar}
\Psi: C^\infty(X)\xrightarrow{\sqrt{-1}\partial\ol{\partial}}\mathcal{E}
\end{equation}
is surjective.  As observed in~\cite{KazTak} $B(X)$ provides a nontrivial cokernel of (\ref{eq:dd-bar}).
In fact, let $\mathbf{L}\in B(X)$, and let $h$ be any Hermitian metric on $\mathbf{L}$.  Then
$\beta_{\mathbf{L}} =\frac{\sqrt{-1}}{2\pi}\Theta_h$ is a real $(1,1)$-form which is exact because $\mathbf{L}$ is
topologically trivial.  But suppose $\beta_{\mathbf{L}} =\sqrt{-1}\partial\ol{\partial} f$.  Then the metric
$h' =e^{2\pi f}h$ has curvature $\frac{\sqrt{-1}}{2\pi}\Theta_{h'} =\beta_{\mathbf{L}} +\sqrt{-1}\ol{\partial}\partial f =0$.
Thus the Chern connection of $h'$ is flat, and since $X$ is simply connected, there is a parallel section $\sigma\in\Gamma(\mathbf{L})$.
Since $\ol{\partial}\sigma =\nabla^{0,1}\sigma =0$, $\sigma$ is holomorphic and $\mathbf{L}$ is trivial, a contradiction.
This defines an injective map
\begin{equation}
B(X)\otimes\Q \hookrightarrow\mathcal{E}/ C^\infty(X).
\end{equation}

If $M\subset X$ is a sufficiently large relatively compact strictly pseudoconvex domain, then $\pi_1(M)=1$.
And from Proposition~\ref{prop:B-isom} and the above we have the
\begin{prop}
There exist infinitely many, topologically distinct, 1-convex surfaces which contain strictly pseudoconvex domains which do not
satisfy the $\partial\ol{\partial}$-lemma but nevertheless satisfy Theorem~\ref{thm:K-E}.
\end{prop}

The above arguments lead to the following more general result which is perhaps worth mentioning.
\begin{prop}
Let $X$ be a complex manifold with $H_1(X,\R)=0$.  Then $X$ satisfies the $\partial\ol{\partial}$-lemma if and only if
$H^1(X,\mathcal{O})=0$.
\end{prop}
\begin{proof}
Denote by $\pic^o X$ the subgroup of $\pic X$ of topologically trivial line bundles.  Then we have
\begin{equation}
\pic^o X =H^1(X,\mathcal{O})/\im\bigl[H^1(X,\Z)\rightarrow H^1(X,\mathcal{O}) \bigr].
\end{equation}
Choose $p\in\Z_+$ so that $p\cdot H_1(X,\Z)=0$.  Let $\mathbf{L}\in\pic^o X$ be an element with $\mathbf{L}^p$ non-trivial.
For any Hermitian metric $h$, $\omega_{\mathbf{L}}=\frac{\sqrt{-1}}{2\pi}\Theta_h \in\mathcal{E}$.
If $\omega_{\mathbf{L}}=\sqrt{-1}\partial\ol{\partial}f$, then there exists a metric $h'$ with a flat connection on $\mathbf{L}$.
But a flat connection corresponds to an $\alpha_{\mathbf{L}}\in\Hom(\pi_1(X), S^1)=\Hom(H_1(X,\Z),S^1)$.
We have $\alpha_{\mathbf{L}}^p =\alpha_{\mathbf{L}^p} =1$, which implies that $\mathbf{L}^n$ is trivial.
\end{proof}

The following shows that a counterexample to Conjecture~\ref{conj:weak-hyp} would have to have singular exceptional set.
\begin{prop}
Let $(M,M')$ be a strictly pseudoconvex finite manifold with K\"{a}hler form $\omega\in -c_1(M)$.  If the exceptional set
$E\subset M$ is smooth, possibly not connected, then $\mathbf{K}_{M'}$ admits an Hermitian metric which is positive
in a neighborhood $\tilde{M}$ of $\ol{M}$.
\end{prop}
\begin{proof}
Let $h$ an Hermitian metric on $\mathbf{K}_{M'}$.  Since each connected component $E_i$ of $E$ is obviously K\"{a}hler,
there exists an $f\in C^{\infty}(M')$ so that the metric $h'=e^f h$ satisfies
\begin{equation}
\begin{split}\label{eq:except-ddbar}
\omega|_{E_i} & =-\frac{\sqrt{-1}}{2\pi}\partial\ol{\partial}\log h'|_{E_i} \\
              & =\bigl(-\frac{\sqrt{-1}}{2\pi}\partial\ol{\partial}\log h -\frac{\sqrt{-1}}{2\pi}\partial\ol{\partial}f\bigr)|_{E_i}.
\end{split}
\end{equation}
We may assume that $M'$ is 1-convex, and let $\pi:M' \rightarrow Y$ be the Remmert reduction.  If $\psi$ is
the pull-back by $\pi$ of a strictly plurisubharmonic function on $Y$, then we have
\begin{align}
\sqrt{-1}\partial\ol{\partial}\psi(x) & \geq 0, \text{for }x\in M'\\ \label{eq:stron-pseu1}
\sqrt{-1}\partial\ol{\partial}\psi(x) & >0,  \text{for } x\in M'\setminus E\\ \label{eq:stron-pseu2}
\sqrt{-1}\partial\ol{\partial}\psi(x) & >0, \text{on } \mathbf{N}_{M'/E}\text{ for }x\in E.\\ \label{eq:stron-pseu3}
\end{align}
Then it is easy to see from (\ref{eq:except-ddbar}), (\ref{eq:stron-pseu1}), (\ref{eq:stron-pseu2}), and (\ref{eq:stron-pseu3}) that
for $A>0$ sufficiently large the metric $e^{-A\psi} h'$ has curvature
\begin{equation}
-\sqrt{-1}\partial\ol{\partial}\log h' +A\sqrt{-1}\partial\ol{\partial}\psi >0,\quad\text{on }\tilde{M},
\end{equation}
where $\tilde{M}$ is a relatively compact neighborhood of $\ol{M}$.
\end{proof}

\section{Normal CR infinity}

\subsection{Consequences of a normal CR infinity}

Assuming that the CR boundary $S=\partial M$ of a strictly pseudoconvex manifold is normal has strong consequences on $M$.
First we mention an embedding result for 1-convex manifolds in~\cite{EtoKazWat} which was generalized to complex spaces in
~\cite{VanTan}.
\begin{thm}\label{thm:1-conv-emb}
Let $X$ be a 1-convex complex space.  Then $X$ is embeddable in $\cps^M \times\C^N$ if and only if there is a positive
holomorphic line bundle on $X$.
\end{thm}
Given a positive line bundle $\mathbf{L}$ on $X$ it shown that there is an $N_0 \in\N$ so that for $k\geq N_0$ there finitely many sections
$s_0,\ldots,s_p \in H^0(X,\mathcal{O}(\mathbf{L}^k))$ so that $\{z\in X : s_0 (z)=\cdots =s_p (z)=0\}$ is empty and the
map $\Psi :X\rightarrow\cps^p$ restricts to an embedding on a neighborhood $U$ of the exceptional set $E$ of the
Remmert reduction $\pi: X\rightarrow Y$.  This is combined with the embedding $\Upsilon:Y\rightarrow\C^n$ of the Stein space $Y$
(cf.~\cite{Nar}) gives an embedding
\begin{equation}
\Psi\times\Upsilon : X\rightarrow \cps^p \times\C^n.
\end{equation}

This gives us another necessary and sufficient condition for a strictly pseudoconvex domain $M$ to admit a K\"{a}hler-Einstein
metric.
\begin{cor}[to Theorem~\ref{thm:K-E}]\label{cor:K-ample}
A strictly pseudoconvex domain $M$ admits a K\"{a}hler-Einstein metric if and only if for some $k\geq 1$ there are finitely many
sections $s_0,\ldots,s_M \in H^0(M,\mathcal{O}(\mathbf{K}_M^k))$ inducing an embedding of a neighborhood $U$ of the exceptional set
of $M$ in $\cps^M$.
\end{cor}

\begin{prop}\label{prop:nor-st-pseudo}
Suppose $M$ is a strictly pseudoconvex manifold such that the induced CR structure on $S=\partial M$ is normal.
Then the Remmert reduction of $M$ is $\hat{M}=C(S)_{r<1}$, where $C(S)_{r<1} =\{(x,r)\in C(S) : r<1\}$ is the domain in the Sasaki cone of $S$, with its induced Sasaki structure.
\end{prop}
\begin{remark}
Note that $C(S)\cup\{o\}$, with the vertex, has a unique structure of a normal Stein variety~\cite{vanCo1}, and also an
affine variety~\cite{vanCo2}.  We will consider $C(S)$ as such with the addition of the vertex.
\end{remark}

In other words, $M\subset X$ is a domain in a resolution $\pi:X\rightarrow C(S)$ of the cone $C(S)$ with exceptional fiber $E=\pi^{-1}(o)$ over $o\in C(S)$.
\begin{proof}
We may suppose that $M\subset X$ with $X$ a 1-convex manifold with Remmert reduction $\pi:X\rightarrow Y$.  Thus $\pi$ maps
$M$ to the strictly pseudoconvex domain $N\subset Y$.  We first prove the following.
\begin{lem}
The action of $\mathfrak{CR}(S,D,J)$ extends to a holomorphic action on $N$.
\end{lem}
\begin{proof}[Proof of Lemma]
Since $Y$ has finitely many isolated singular points, it has finite embedding dimension and there there is an embedding
$\iota: Y\rightarrow\C^N$ (cf.~\cite{Nar}).  Let $\psi\in\mathfrak{CR}(S,D,J)$, and define $f^\psi_j =\psi^*(z_j \circ\iota),
\ j=1,\ldots,N.$  We have $\ol{\partial}_b f^\psi_j =0$, i.e. $f^\psi_j $ is annihilated by $D^{0,1}\subset D\otimes\C\subset TS\otimes\C$,
so by the extension theorem of J. Kohn and H. Rossi~\cite{KohnRoss} the $f^\psi_j $ extend to holomorphic functions on
$\ol{M}$.  There are holomorphic functions $g^\phi_j :\ol{N} \rightarrow\C^N$ with $g^\phi_j \circ\pi =f^\psi_j$.
Denote $F^\psi :=(f^\psi_1 ,\ldots,f^\psi_N)$ and $G^\psi =(g^\psi_1,\ldots, g^\psi_N )$.  So $F^\psi :\ol{M}\rightarrow\C^N$
and $G^\psi :\ol{N}\rightarrow\C^N$ with $G^\psi \circ\pi = F^\psi$.

We have $\im F^\psi|_S =\im\iota|_S$.  If $h\in\mathcal{O}_{\C^N}(U)$ is any function defined with $U\cap\im\iota|_S \neq\emptyset$
with $U$ connected and vanishing on $\im\iota\subset\C^N$, then $h\circ F^\psi$ vanishes on $S\cap (F^\psi)^{-1}(U)$ so vanishes identically.  Therefore there is a neighborhood $V\subset X$ of $S$ with $F^\psi (V)\subset\im\iota$.  Now cover $F^\psi (\ol{M})$ with finitely
many neighborhoods $U_\alpha,\ \alpha=1,\ldots, m,$ for which $\im \iota \cap U_\alpha =\{h^\alpha _1 =\cdots=h^\alpha_{k_\alpha} =0\}$ for
defining functions $h^\alpha_i \in\mathcal{O}(U_\alpha)$.  Let $V_\alpha =(F^\psi )^{-1}(U_\alpha )$.  If
$V_\alpha\cap V\neq\emptyset$, then $h^\alpha_i \circ F^\psi$ vanish on $V_\alpha$.  Thus $F^\psi (V\cup V_\alpha)\subset \im\iota$.
Continuing this argument shows that $F^\psi(\ol{M})\subset\im\iota$.

Since $\iota$ maps $Y$ biholomorphically onto its image, we can define $\mu^\psi :\ol{N}\rightarrow Y$ by $\iota^{-1} \circ G^\psi$.
Let $\phi$ be a plurisubharmonic defining function of $N\subset Y$, i.e. $N=\{\phi<0\}$.  Then $(\pi\circ\mu^\psi )^* \phi$ is
plurisubharmonic and takes the value 0 on $S=\partial N$, so $(\pi\circ\mu^\psi)^* \phi (x) <0$ for $x\in M$ by the maximum principle.
Therefore we have $\mu^\psi :\ol{N}\rightarrow\ol{N}$.

Suppose $\psi_1 ,\psi_2 \in\mathfrak{CR}(S,D,J)$.  Since $\iota\circ\mu^{\psi_1}\circ\mu^{\psi_2}\circ\pi -\iota\circ\mu^{\psi_1 \circ\psi_2}\circ\pi$ vanishes on $S$ it must vanish identically.  Therefore $\mu^{\psi_1 \circ\psi_2} =\mu^{\psi_1}\circ\mu^{\psi_2}$,
so $\mathfrak{CR}(S,D,J)$ acts on $\ol{N}$ by biholomorphisms.
\end{proof}

Suppose $\xi\in\in\mathfrak{t}^+_k \subseteq\mathfrak{cr}^+(S,D,J)$.  Then we can replace $\xi$ with an integral element in
$Z_k^+ \subset\mathfrak{t}^+_k$.  This $\xi$ generates an $\operatorname{U}(1)$-action on $\ol{N}$.  If $\phi$ is a defining function for $M\subset X$,
$0< \eta(\xi)=-\frac{1}{2}d\phi(J\xi)$, so $J\xi$ points inward at $S$.  If $\epsilon<0$, then
\begin{equation}\label{eq:bound-diff}
(\epsilon,0] \times S \ni (t,x)\longrightarrow \exp(-tJ\xi)x
\end{equation}
is a diffeomorphism onto a neighborhood of $S$ in $\ol{N}$.  We consider a complex structure on $\R\times S$ which is that
of (\ref{eq:cone-cst}) in the coordinate $t=\log r$ of $\R$; that is,
\begin{equation}\label{eq:cyl-cst}
I(X):= \Phi(X) -\eta(X)\partial_t, \quad I(\partial_t) =\xi,
\end{equation}
for $X\in TS$.  It is not difficult to see that (\ref{eq:bound-diff}) is a biholomorphism between $((\epsilon,0]\times S ,I)$
and $(V,J)$ where $V$ is a neighborhood of $S$ in $\ol{N}$.

Since $C(S)_{r<1}$ and $N$ are both normal Stein spaces Hartogs' theorem implies that an holomorphic function on
$((\epsilon,0)\times S ,I)$ extends to $C(S)$ and likewise for holomorphic functions on $V\setminus S\subset N$.
Therefore $\mathcal{O}(C(S)_{r<1})\cong\mathcal{O}(N)$, and we have a biholomorphism $C(S)_{r<1} \cong N$.
\end{proof}

\begin{cor}\label{cor:Stein-dom}
If $M$ is a strictly pseudoconvex Stein domain with the CR structure on $S=\partial M$ normal, then $S=\mathbb{S}^{2n-1}$, the
$2n-1$ sphere with a transversal deformation of the standard CR structure.
\end{cor}
See Section~\ref{subsubsect:transv} for an explanation of ``transversal deformation.''
\begin{proof}
Since $M$ is a domain in the resolution $\pi: X\rightarrow C(S)$, we must have $\pi^{-1}(o) =E=\empty$.  Thus $M\subset C(S)$ and
$C(S)$ is nonsingular.  Suppose $\xi\in Z_k^+ \subset\mathfrak{t}^+_k$.  We have an action
$\iota:\operatorname{U}(1)\rightarrow\Aut\bigl(T_o C(S)\bigr)$, with weights $(w_1,\ldots,w_n)\in\Z^n_{>0}$.
One can get an equivariant coordinate system $(U, z_1,\ldots,z_n)$, i.e. $\iota(u)(z_1,\ldots,z_n)=(u^{w_1} z_1,\ldots,u^{w_n} z_n)$,
by a simple averaging argument.  Since $\exp(tJ\xi)$ maps $M$ into $U$ for large enough $t>0$, we have $U\cong C(S)\cong\C^n$.
We have $\mathbb{S}^{2n-1} \subset\C^n$ with the standard CR structure $(D,J)$ and $\xi$ is a CR Reeb vector field.
For $z\in S$ there is a unique $t\in\R$ with $(e^{w_1 t} z_1, \ldots,e^{w_n t} z_n)\in\mathbb{S}^{2n-1}$, and if we define $\psi(z)=t$
Then if we set $r' =e^{-\psi} r$, $\mathbb{S}^{2n-1}=\{r'=1\}$.  Note that $r'\partial_{r'} =r\partial_r$, so the Euler vector field  and also the Reeb vector field is unchanged.  And the contact forms are related by
\begin{equation}
\eta' =2d^c \log r' =2d^c \log r -2d^c \psi =\eta-2d^c \psi.
\end{equation}
\end{proof}

\begin{prop}
Let $M\subset X$ be a strictly pseudoconvex domain in one of the 1-convex surfaces of Section~\ref{sec:com-dd-bar} so that
$\pi_1(M)=1$.  Then the CR structure on $S=\partial M$ is not normal.
\end{prop}
\begin{proof}
Suppose otherwise, then the Remmert reduction is $\pi:M\rightarrow\hat{M}=\{r<1\}\subset C(S)$.  Let $\mathbf{L}\in B(M)=\Z^{2g}$
be a non-trivial element.  By definition $\mathbf{L}$ is holomorphically trivial in a neighborhood of the exceptional curve $E$.
So there is an $\hat{\mathbf{L}}$ on $\hat{M}$ with $\pi^*\hat{\mathbf{L}}=\mathbf{L}$.  But obviously, $H^2(\hat{M},\Z)=0$,
so $\pic\hat{M} =0$ implying that $\mathbf{L}$ is holomorphically trivial.

This proposition also simply follows from~\cite[Prop. 1]{Colt} which give as a condition for $B(M)=0$ that
$\pi^* H^2(\hat{M},\Z)\rightarrow H^2(M,\Z)$ is injective.
\end{proof}

This proposition has an obvious generalization.
\begin{cor}
Let $M\subset X$ be a strictly pseudoconvex domain.  If $B(M)\neq 0$, then the CR structure on $\partial M$ is not normal.
\end{cor}

\begin{defn}
A strictly pseudoconvex finite manifold $M$ is called a \emph{normal} K\"{a}hler-Einstein manifold if it has a complete K\"{a}hler-Einstein metric and its conformal CR infinity is normal.
\end{defn}

It is well known~\cite{OrnVerb}, see also~\cite{vanCo2}, that the K\"{a}hler cone $C(S)\cup\{o\}$ over a Sasaki manifold $S$ is an affine
variety.  The homomorphism $\iota :\operatorname{U}(1) \rightarrow\mathfrak{CR}(S,D,J)$ with $\iota_*\frac{\partial}{\partial\theta}=\xi\in\mathfrak{cr}^+(D,J)$
extends to $\iota:\C^* \rightarrow\mathfrak{Hol}(C(S)\cup\{o\})$, an algebraic action.
The following result is evident when this is combined with Theorem~\ref{thm:1-conv-emb}.
\begin{prop}
A normal K\"{a}hler-Einstein manifold is a domain in a quasi-projective variety.
\end{prop}
\begin{proof}
We have that $M\subset X$ where $X$ is a resolution of $C(S)\cup\{o\}$ with $S=\partial M$.  We compactify $X$ to $\hat{X}$ as follows.
Choose $\xi\in Z^+_k \subset\mathfrak{t}^+_k$.   Then $C(S)$ is biholomorphic to $\mathbf{L}^\times$ where $\mathbf{L}\rightarrow W$
is an holomorphic orbibundle over $W$.  Now define an orbifold $\hat{X}$ by replacing $C(S)\subset X$ with the total space of
$\pi:\mathbf{L}^{-1} \rightarrow W$, so $\hat{X}$ is $X$ with the divisor $W$ with positive normal bundle $\mathbf{L}^{-1} =[W]|_W$ added.
(The divisor $W\subset\hat{X}$ is only $\Q$-Cartier, but the following arguments work in the orbifold setting.  See~\cite{BoyGal}.)
Let $\sigma\in H^0(\hat{X},\mathcal{O}(W))$ be a section vanishing on $W$.  Let $h$ be a metric on $[W]$ with
$\frac{\sqrt{-1}}{2\pi}\Theta_h |_W >0$.  In a neighborhood of $W$ identified with a neighborhood of the zero section of
$\pi : \mathbf{L}^{-1}\rightarrow W$ define $\tilde{h} =e^{-\pi^*h|z|^2} h$, where $z$ is the fiber coordinate, and extend $\tilde{h}$
to all of $[W]$.  Then $-\log\tilde{h}|\sigma|^2$ is strictly plurisubharmonic near $W$.  Since
$-\log\tilde{h}|\sigma|^2\rightarrow\infty$ at $W$, we can modify it away from $W$ to a plurisubharmonic function $f$ on $X$.
If we set $q=-f-\log\tilde{h}|\sigma|^2$, then $\hat{h} =e^q \tilde{h}$ is a metric on $[W]$ with
$\frac{\sqrt{-1}}{2\pi}\Theta_{\hat{h}}\geq 0$ and $\frac{\sqrt{-1}}{2\pi}\Theta_{\hat{h}}>0$ in a neighborhood $N$ of $W$.  Then for sufficiently large $k>0$, $\mathbf{F}=\mathbf{K}_{\hat{X}}\otimes [kW]$ admits a metric with positive curvature.
By the Baily embedding theorem~\cite{Bail} for sufficiently large $p$, the sections of $\mathbf{F}^p$ define an embedding
$\psi_{\mathbf{F}^p} :\hat{X} \rightarrow\cps^N$.
\end{proof}

\begin{thm}
Let $M$ be a normal K\"{a}hler-Einstein manifold with CR infinity $(S,D,J)$.  Then the action of the connected component
$\mathfrak{CR}_0(S,D,J)$ extends to a holomorphic action on $M$.  Thus there is an injection
\begin{equation}
\iota : \mathfrak{CR}_0(S,D,J) \hookrightarrow\mathfrak{Hol}(M)=\mathfrak{Isom}(M,g,J)
\end{equation}
\end{thm}
\begin{proof}
Let $\mathbf{F}^p=\mathbf{K}^p_{\hat{X}}\otimes [kpW]$ be the very ample bundle as above.
Choose any $\xi\in Z^+_k \subset\mathfrak{t}^+_k$, then as above the action of $\C^*$ on $C(S)$ is algebraic.
Let $\sigma\in H^0(\hat{X},\mathcal{O}(\mathbf{F}^p))$, so $\sigma|_{C(S)} \in H^0(C(S),\mathcal{O}(\mathbf{K}_{C(S)}^p))$.
If $\alpha: C(S)\times\C^* \rightarrow C(S)$ is the projection, then we have a section $\tilde{\sigma}$ of $\alpha^* \mathbf{K}_{C(S)}^p$
$\tilde{\sigma}=\iota(z)^*\sigma|_{C(S)}$, $z\in\C^*$.  And $\tilde{\sigma}$ is a rational section of
$\alpha^*\mathbf{K}^p_{\hat{X}}$, where $\alpha:\hat{X}\times\C^* \rightarrow\hat{X}$ is again the projection.
We know that $\tilde{\sigma}$ is regular except perhaps on some $E_i \times\C^*$ where $E_i\subset E$ is an exceptional divisor of $X$ or
on $W\times\C^*$.  But since $\tilde{\sigma}|_{X\times\{1\}}$ is smooth, it is easy to see that $\tilde{\sigma}$ is regular
along $E\times\C^*$.  It is easy to see that $\tilde{\sigma}$ has at most a pole of order $kp$ along $W\times\C^*$.
So $\iota(z)^*\sigma\in H^0(\hat{X},\mathcal{O}(\mathbf{F}^p))$ for $z\in\C^*$.

As above we have the embedding
\begin{equation}
\psi_{\mathbf{F}^p} :\hat{X} \rightarrow\cps^N=\mathbb{P}(H^0(\hat{X},\mathcal{O}(\mathbf{F}^p))^*).
\end{equation}
If we denote the action of $\C^*$ on $\cps^N$ by $\iota(z)$, then for $x\in\hat{X}\setminus E$
$\iota(z)(\psi_{\mathbf{F}^p}(x) =\psi_{\mathbf{F}^p}(\iota(z)(x))$.  So $\C^*$ preserves
$\psi_{\mathbf{F}^p}(\hat{X}\setminus E) \subset\cps^N$.  Since $E$ is nowhere dense, it is easy to see that
$\psi_{\mathbf{F}^p}(\hat{X})\subset\cps^N$ is preserved by the action of $\C^*$ on $\cps^N$.

For each $\xi\in Z^+_k$ the $\operatorname{U}(1)$ action extends to $X$.  These actions generate $T_k \subset\mathfrak{CR}(S,D,J)$.
Since all maximal tori are conjugate, $Z^+_k$ generates $\mathfrak{t}_k$ for all maximal tori.
So the action of every maximal torus $T_k \subset\mathfrak{CR}(S,D,J)$ extends to $M$.
When $\mathfrak{CR}(S,D,J)$ is compact every element of $\mathfrak{CR}_0(S,D,J)$ is contained in a maximal torus so the result
follows.  Otherwise, $(S,D,J)$ is the round sphere $\mathfrak{CR}(\mathbb{S}^{2n-1},D,J)=\operatorname{PSU}(1,n)$ and $M=\mathcal{H}^n_{\C}$
the complex hyperbolic space, and the theorem follows.
\end{proof}

\subsection{Transversal deformations}\label{subsubsect:transv}

Every normal K\"{a}hler-Einstein manifold $M$, with fixed CR Reeb vector field, has a natural infinite dimensional space of deformations
parameterized by basic functions on $S=\partial M$ with sufficiently small 2nd derivatives.  By basic we mean invariant under the Reeb
action of $\xi$.  These correspond to \emph{transversal deformations} of the CR structure on $S$.

As above $S=\{r=1\}\subset C(S)$, so we may take $r^2 -1$ as the defining function of $M$.
The CR distribution $D=\ker\eta$, where $\eta=2d^c \log r$.  Here $\eta$ is the unique 1-form with $\ker\eta =D$ and $\xi\contr\eta=1$.

If $\psi\in C^\infty_B(S)$ is a basic function, which we may take as a function on $C(S)$, then set $r'=e^\psi r$.
Then $r' =1$ defines the boundary $S'$ of a domain in $C(S)$ with defining function $r'^2 -1$.
One can check that $S'$ is naturally diffeomorphic to $S$, one has $\xi\contr\eta'=1$, and this alters the CR structure on $S$ by
\begin{gather}
\eta' =\eta +2d^c \psi \label{eq:cont-def}\\
\Phi' =\Phi -\xi\otimes\eta'\circ\Phi.\label{eq:compl-def}
\end{gather}
If $\eta' \wedge (d\eta')^{n-1}$ is nowhere zero, this defines another normal strictly pseudoconvex CR structure on $S$.
This is equivalent to the new Levi form $\mathbf{L}^{D'}=d\eta'$ being positive.

\begin{figure}[tbh]
 \centering
 \includegraphics[scale=1]{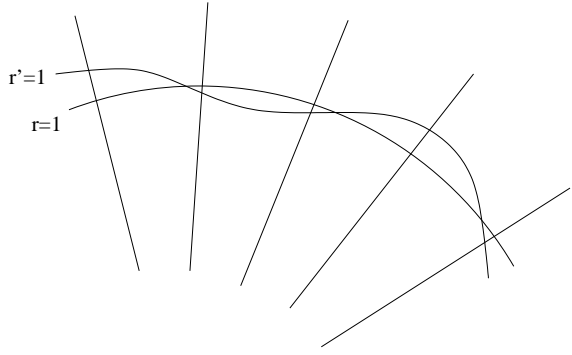}
 \caption{Transversal deformation}
 \label{fig:transverse}
\end{figure}

Equations (\ref{eq:cont-def}) and (\ref{eq:compl-def}) define the deformed CR structure on $S$.  Alternatively, one may fix
the complex structure on $C(S)$ deform the domain via $r'=e^\psi r$.  Figure~\ref{fig:transverse} shows the respective domains
$C(S)_{r<1}$ and $C(S)_{r'<1}$.

Another type of deformation will be of interest in the final section which we call \emph{transversal deformations of the second kind}.
Let $\alpha\in H^1_B(S)$, then we have the CR structure on $S$
\begin{equation}\label{eq:cont-def2}
\eta' =\eta +\alpha,
\end{equation}
with $\Phi$ given again by (\ref{eq:compl-def}).  If $\alpha =df$, with $f\in C_B^\infty (S)$, then the new CR structure is just a gauge
transformation of $(D,J)$ in the Reeb direction.

\subsection{possible CR infinities}

We consider the problem of which strictly pseudoconvex CR manifolds $S$ are CR infinities of some K\"{a}hler-Einstein
manifold.
\begin{prob}
Which strictly pseudoconvex CR manifolds $S$ are conformal CR infinities of complete K\"{a}hler-Einstein manifolds.
\end{prob}
By Theorem~\ref{thm:K-E} this is equivalent to whether there exists a complex manifold $M$ with $\partial M=S$ and
$\mathbf{K}_M$ positive.  One necessary condition is that $S$ is embeddable, meaning there is a smooth embedding
$S\hookrightarrow\C^N$ with the CR structure on $S$ induced from the complex structure of $\C^N$.
Conversely by a theorem of Harvey and Lawson~\cite{HarLaw} $S=\partial M$ for some $M$ if $S$ is embeddable.
It is known~\cite{BdeMon} that all strictly pseudoconvex $(S,D,J)$ are embeddable in dimension $2n-1 \geq 5$.
But generic perturbations of the standard CR structure on $\mathbb{S}^3$ are nonembeddable.

We consider the more restricted Problem~\ref{prob:bound-K-E}, considering only normal CR structures.
We can prove a negative result in case $S$ has a normal CR structure $(D,J)$.  We will need a definition.
\begin{defn}
We say that a normal CR manifold $(S,D,J)$ has \emph{property S-E} if $c_1(D,J)=0$ and for $\xi\in\mathfrak{cr}^+(D,J)$
the transversal first Chern class $c_1^B >0$, i.e. represented by a positive $(1,1)$-form.
\end{defn}
If this property holds for some $\xi\in\mathfrak{cr}^+(D,J)$, then it holds for all of $\mathfrak{cr}^+(D,J)$, so it is
intrinsic to $(D,J)$.  If $\xi'\in\mathfrak{cr}^+(D,J)$, then by conjugation with an element of $\mathfrak{CR}(D,J)$
we may suppose $\xi,\xi' \in\mathfrak{t}^+$ where $\mathfrak{t}$ is the Lie algebra of a maximal torus of $\mathfrak{CR}(D,J)$.
The respective contact forms satisfy $\eta' =f\eta$ where $f=(\eta(\xi'))^{-1}$, so
\begin{equation}\label{eq:con-Kah}
(d\eta')^{n-1} |_D =f^{n-1}(d\eta)^{n-1}|_D.
\end{equation}
Recall that $\frac{1}{2\pi}\ric^T(\frac{1}{2}d\eta)\in c_1^B$, where $c_1^B$ is with respect to the foliation generated by $\xi$.
Thus from (\ref{eq:con-Kah}) with respect to the foliation of $\xi'$
\begin{equation}
c_1^B \ni \frac{1}{2\pi}\ric^T (\frac{1}{2}d\eta')=\frac{1}{2\pi}\ric^T (\frac{1}{2}d\eta) -\frac{\sqrt{-1}}{2\pi}\partial\ol{\partial}\log f^{n-1}.
\end{equation}
Here the Ricci forms are computed on $(D,J)$.

This is precisely the topological condition for the associated Sasaki structure $(S,g,\eta,\xi,\Phi)$ to
admit a possible transversal deformation to a Sasaki-Einstein structure.  Sasaki manifolds satisfying this condition
have been studied extensively (cf.~\cite{GauMarSpaWal1,GauMarSpaWal1,BoyGalKol,BoyGalKolThom,FutOnoWang},) mainly in order to construct new Einstein manifolds.

In terms of the K\"{a}hler cone $C(S)$ property S-E is equivalent to the existence of an holomorphic $(n,0)$-form
$\Omega$ with $\mathcal{L}_\xi \Omega =a\sqrt{-1}\Omega$, with $a>0$, and satisfying
\begin{equation}\label{eq:hol-form}
\left(\frac{i}{2}\right)^n (-1)^{\frac{n(n-1)}{2}}\Omega\wedge\ol{\Omega} =e^h \frac{1}{n!}\omega^n,
\end{equation}
where $\omega$ is the natural K\"{a}hler metric on $C(S)$ for the Sasaki manifold $(S,\frac{n}{a}\xi, D,J)$ and $h$ is invariant under the action generated by $\xi$ and $r\partial_r$.  See~\cite{FutOnoWang}.
If $S$ is not simply connected then one may have to take $\Omega$ to be multivalued, i.e. it defines a section of
$\mathbf{K}_{C(S)}^{\otimes p}$ for some $p\in\N$.  We have $e^h \omega^n$ defining an Hermitian metric on $\mathbf{K}_{C(S)}$
via (\ref{eq:hol-form})
\begin{equation}
\|\Psi\|^2 =\left(\frac{i}{2}\right)^n (-1)^{\frac{n(n-1)}{2}}e^{-h}\frac{\Psi\wedge\ol{\Psi}}{\omega^n},
\end{equation}
for $(n,0)$-form $\Psi$.  It is easy to see that the associated connection $\nabla^h$
on $\mathbf{K}_{C(S)}$ is flat.  And we have the surjective homomorphism to the holonomy group $\pi_1(S)\rightarrow\Hol (\nabla^h)\subset \operatorname{U}(1)$,
whose image is a finite group $\Z_p$.  If $p=1$, then the singularity $o\in C(S)\cup\{o\}$ is \emph{Gorenstein}.
And if $p>1$, $o\in C(S)\cup\{o\}$ is \emph{$\Q$-Gorenstein}.

\begin{thm}\label{thm:poss-inf}
Let $(S,D,J)$ be a normal strictly pseudoconvex CR manifold with property S-E such that the singularity $C(S)\cup\{o\}$ is Gorenstein, e.g. $\pi_1(S)=e$.  Then $(S,D,J)$ is the CR infinity of a complete K\"{a}hler-Einstein manifold $M$ only if
$S=\mathbb{S}^{2n-1}$, with $(D,J)$ a transversal deformation of the standard CR structure and $M\subset\C^n$.
\end{thm}
\begin{remark}
One can show if $(S,D,J)$ has property S-E, then $\pi_1(S)$ must be finite.  This follows from the existence of a Sasaki structure with
positive Ricci curvature and an application of Meyer's theorem.  In most applications one will have $(S,D,J)$ simply connected.
\end{remark}
\begin{proof}
By Proposition~\ref{prop:nor-st-pseudo} $M$ is a resolution $\pi:M\rightarrow C(S)_{r<1}$.
We have a holomorphic n-form $\Omega$ on $C(S)$ satisfying (\ref{eq:hol-form}).  It is a result of~\cite{Lau} and~\cite{Bur}
that $o\in C(S)$ is a rational singularity if and only if it has a small neighborhood $U$ with
\begin{equation}\label{eq:rational}
\int_U \Omega\wedge\ol{\Omega} <\infty.
\end{equation}
And this easily follows from (\ref{eq:hol-form}).  Moreover, it is a consequence of (\ref{eq:rational}) that for any
resolution $\pi:M\rightarrow C(S)_{r<1}$ the form $\Omega$ extends to a holomorphic form on $M$ (cf.~\cite{Lau}).
We recall a definition.
\begin{defn}
Let $X$ be a normal $\Q$-Gorenstein variety.  Then $X$ has \emph{canonical singularities} if for every resolution
$\pi :\hat{X}\rightarrow X$ one has
\begin{equation}\label{eq:can-sing}
K_{\hat{X}} =\pi^* K_X + \sum_{i} a_i E_i,
\end{equation}
with $a_i \geq 0$ for each exceptional divisor $E_i$.  The equality in (\ref{eq:can-sing}) means linear equivalence.
\end{defn}
It is sufficient to check (\ref{eq:can-sing}) for one resolution, and we have each $a_i \geq 0$.
By Corollary~\ref{cor:K-ample} for some $q\geq 1$ $\mathbf{K}_M^q$ is ample in a neighborhood of the exceptional set $E$, and
we have $qK_M =q\sum_{i} a_i E_i$.  If $\sigma\in H^0(M,\mathcal{O}(\mathbf{K}_M^q))$, then $f=\frac{\sigma}{\Omega^q}$ is a
meromorphic function which is holomorphic on $C(S)_{r<1} \setminus\{o\}$.  By the Riemann extension theorem $f$ extends holomorphically
to $\tilde{f}$ on $C(S)_{r<1}$.  So $f=\pi^* \tilde{f}$, and $\sigma|_E$ is a constant multiple of $\Omega^q |_E$.
Therefore, we must have $E=\emptyset$, $M =C(S)_{r<1}$.  Since $M$ is smooth we must have $M\subset\C^n$ and
the rest follows as in the proof of Corollary~\ref{cor:Stein-dom}.
\end{proof}

\begin{remark}
Theorem~\ref{thm:poss-inf} is a global result.  If $(S,D,J)$ satisfies the assumptions of the theorem and in addition
is Sasaki-Einstein, then the Ansatz in Section~\ref{subsect:diag-quot}, due to E. Calabi, constructs an incomplete
K\"{a}hler-Einstein metric with CR infinity $(S,D,J)$.

Note also that the examples in \ref{subsect:Hirz-Jung} and \ref{subsect:diag-quot} show that Gorenstein assumption in the theorem
in necessary.
\end{remark}

\section{Examples}

We consider some cases in which Theorem~\ref{thm:K-E} is easily applicable.  The following easy result will be helpful
in some of the cases that follow.
\begin{prop}\label{prop:quasi-proj}
Suppose $X$ is a projective manifold and $W\subset X$ is a smooth divisor ($X$ may have orbifold singularities along $W$) with $\mathbf{K}_X \otimes [kW] >0$, for some $k\geq 1$, and $[W]|_W >0$.  Then $X\setminus W$ is 1-convex and
$\mathbf{K}_{X\setminus W} >0$.
\end{prop}

\subsection{Negative bundles}

Let $\pi:\mathbf{E}\rightarrow N$ be an holomorphic bundle with a Hermitian metric $h$.  Recall, that $\mathbf{E}$ has a unique
Chern connection $\nabla$ which is compatible with $h$ and $\nabla^{0,1} =\ol{\partial}$.  In a holomorphic local frame
$(e_1,\ldots,e_r)$ the connection form is $\theta =\partial hh^{-1}$.  We have the curvature
$\Theta\in\Omega^{1,1}(\Hom(\mathbf{E},\mathbf{E})$ given by
\begin{equation}
\Theta_i^j =d\theta_i^j +\theta_k^j\wedge\theta_i^k.
\end{equation}

\begin{defn}\label{defn:pos}
A connection on $\mathbf{E}$ has \emph{positive (resp. negative) curvature} if for each $x\in N$ and for all nonzero
$v\in\mathbf{E}_x$  $\sqrt{-1} h(\Theta_x v,v)$ is a positive (resp. negative) $(1,1)$-form.

A holomorphic bundle $\mathbf{E}$ is \emph{positive (resp. negative)} if it admits a metric whose Chern connection
has positive (resp. negative) curvature.
\end{defn}

This condition was called weakly positive by P. Griffiths~\cite{Grif}, as it is not strong enough to ensure the properties of a
positive line bundle such as Kodaira vanishing.  Although, a weaker condition than in~\ref{defn:pos} was called weakly positive
by H. Grauert~\cite{Grau}.

Define a smooth function $r^2 :=h(v,v)$ on the total space of $\mathbf{E}$.
\begin{prop}\label{prop:neg-pseudo}
Suppose the metric $h$ has negative curvature.  Then the disk bundles $\{r^2 <c\}$, for $c>0$, are strictly pseudoconvex.
In fact, $dd^c r^2$ is a positive $(1,1)$-form outside the zero section.
\end{prop}
\begin{proof}
Choose a local holomorphic frame $(e_1,\ldots,e_r)$, with fiber coordinates $(w_1,\ldots, w_r)$, so that
$\theta=\partial h h^{-1} =0$ at $x\in N$. Then at $x\in N$ we have
\begin{equation}\label{eq:neg-cur}
\Theta_x =d\theta =\ol{\partial}(\partial h h^{-1})=(\ol{\partial}\partial h)h^{-1}.
\end{equation}
While
\begin{equation}\label{eq:psh-r}
\partial\ol{\partial} r^2 =\partial\ol{\partial}h_{i\ol{\jmath}}w_i \ol{w}_{\ol{\jmath}} +dw_i \wedge\ol{\partial}h_{i\ol{\jmath}}\ol{w}_{\ol{\jmath}} +\partial h_{i\ol{\jmath}}w_i \wedge d\ol{w}_{\ol{\jmath}} +h_{i\ol{\jmath}}dw_i \wedge d\ol{w}_{\ol{\jmath}}.
\end{equation}
The two middle terms on the right of (\ref{eq:psh-r}) vanish.  Thus (\ref{eq:neg-cur}) and (\ref{eq:psh-r}) show that at
$(w_r,\ldots, w_r)\in\mathbf{E}_x$ we have $\partial\ol{\partial}r^2 >0$.
\end{proof}

If $\mathbf{E}$ is an arbitrary holomorphic vector bundle and $\mathbf{L}$ is a negative line bundle on $N$,
then the curvature of $\mathbf{E}\otimes\mathbf{L}^\mu$ is
\begin{equation}
\Theta_{\mathbf{E}\otimes\mathbf{L}^\mu} =\Theta_{\mathbf{E}} +\mu\Theta_{\mathbf{L}}.
\end{equation}
Thus for sufficiently large $\mu\gg 0$ $\mathbf{E}\otimes\mathbf{L}^\mu$ is negative.

Associated to a vector bundle $\pi:\mathbf{E}\rightarrow N$ is the bundle of projective spaces
$\tilde{\pi}:\mathbb{P}(\mathbf{E})\rightarrow N$ with fibers $\tilde{\pi}^{-1}(x)=\mathbb{P}(\mathbf{E}_x)$.
Let $\rho:\mathbf{L}\rightarrow\mathbb{P}(\mathbf{E})$ be the universal bundle of lines in $\mathbf{E}$.
The Hermitian metric $h$ on $\mathbf{E}$ defines a natural metric on $\mathbf{L}$.  We will compute the
curvature $\Theta_{\mathbf{L}}$ of this metric on $\mathbf{L}$ in terms of the curvature $\Theta_{\mathbf{E}}$
of $\mathbf{E}$.  This was proved in~\cite{Grif}.  We prove it here as it is important to what follows.

Let $(e_1,\ldots,e_r)$ be a local holomorphic frame of $\mathbf{E}$, and let $\xi=(\xi_1,\ldots,\xi_r)\in\C^r \setminus\{0\}$ be fiber coordinates.  Then denote
$\langle\xi,\xi\rangle:=\sum_{i,j}\xi_i h_{i\ol{\jmath}}\ol{\xi}_{\ol{\jmath}}$.  We have
\begin{equation}\label{eq:curv-her}
\begin{split}
\Theta_{\mathbf{L}}=\ol{\partial}\partial\log\langle\xi,\xi\rangle
  & =\frac{\partial\xi_i h_{i\ol{\jmath}}\ol{\partial}\ol{\xi}_{\ol{\jmath}}-\partial\xi_i \ol{\partial}h_{i\ol{\jmath}}\ol{\xi}_{\ol{\jmath}}-\xi_i \partial h_{i\ol{\jmath}}\ol{\partial}\ol{\xi}_{\ol{\jmath}}+\xi_i \ol{\partial}\partial h_{i\ol{\jmath}}\ol{\xi}_{\ol{\jmath}}}{\langle\xi,\xi\rangle} \\
                    & +\frac{(\partial\xi_i h_{i\ol{\jmath}}\ol{\xi}_{\ol{\jmath}}+\xi_i\partial h_{i\ol{\jmath}}\ol{\xi}_{\ol{\jmath}})\wedge(\xi_i h_{i\ol{\jmath}}\ol{\partial}\ol{\xi}_{\ol{\jmath}}+\xi_i\ol{\partial} h_{i\ol{\jmath}}\ol{\xi}_{\ol{\jmath}} )}{\langle\xi,\xi\rangle^2}.\\
\end{split}
\end{equation}
Rearranging terms in (\ref{eq:curv-her}) we get
\begin{equation}
\Theta_{\mathbf{L}}=\frac{\langle\Theta_{\mathbf{E}}\xi,\xi\rangle}{\langle\xi,\xi\rangle} -P(x,\xi)+Q(x,\xi),
\end{equation}
where
\begin{equation}
P(x,\xi) =\frac{\langle\partial\xi,\partial\xi\rangle}{\langle\xi,\xi\rangle} -\frac{\langle\partial\xi,\xi\rangle\wedge\langle\xi,\partial\xi\rangle}{\langle\xi,\xi\rangle^2}
\end{equation}
is the Fubini-Study metric on the fibers and
\begin{equation}
Q(x,\xi) =\frac{-\langle\theta\xi,\theta\xi\rangle-2\sqrt{-1}\im\langle\partial\xi,\theta\xi\rangle}{\langle\xi,\xi\rangle}
+\frac{2\sqrt{-1}\im(\langle\partial\xi,\xi\rangle\langle\xi,\theta\xi\rangle)+\langle\theta\xi,\xi\rangle\wedge\overline{\langle\theta\xi,\xi\rangle}}{\langle\xi,\xi\rangle^2}.
\end{equation}
One can choose a frame $(e_1,\ldots,e_r)$ so that the connection form $\theta$ vanishes at $x_0 \in N$.
Then $Q(x_0 ,\xi)=0$, so we have the following.
\begin{prop}\label{prop:neg-taut}
If $\mathbf{E}$ is a negative vector bundle, then $\rho:\mathbf{L}\rightarrow\mathbb{P}(\mathbf{E})$ is negative.
\end{prop}

Let $\pi:\mathbf{E}\rightarrow N$ be a rank $r$ bundle with associated projective bundle
$\tilde{\pi}:\mathbb{P}(\mathbf{E})\rightarrow N$ and tautological line bundle $\rho:\mathbf{L}\rightarrow\mathbb{P}(\mathbf{E})$, then the canonical bundle of $\mathbb{P}(\mathbf{E})$ is given by
\begin{equation}\label{eq:can-bund}
\mathbf{K}_{\mathbb{P}(\mathbf{E})}=\tilde{\pi}^*\mathbf{K}_N \otimes\tilde{\pi}^* \det(\mathbf{E})^{-1} \otimes\mathbf{L}^r.
\end{equation}

Consider the compactification $X=\mathbb{P}(\mathbf{E}\oplus\C)$ of $\mathbf{E}$.  So we have
\begin{equation}\label{eq:can-bund2}
\mathbf{K}_X =\tilde{\pi}^* \mathbf{K}_N \otimes\tilde{\pi}^* \det(\mathbf{E})^{-1} \otimes\mathbf{L}^{r+1}.
\end{equation}
Let $D_\infty \subset X$ be the divisor at infinity, that is
$D_\infty =\{[v:0]\in \mathbb{P}(\mathbf{E}\oplus\C)\}=\mathbb{P}(\mathbf{E})$.
Clearly, $[D_\infty]$ and $\mathbf{L}^{-1}$ restrict to the hyperplane bundle on each fiber $\tilde{\pi}^{-1}(x)=\mathbb{P}(\mathbf{E}\oplus\C)_x$.  Thus $[D_\infty]\otimes\mathbf{L} =\tilde{\pi}^*(\mathbf{F})$ for
a line bundle $\mathbf{F}$ on $N$.  One can check that the normal bundle
$\mathcal{N}_{D_\infty} \cong [D_\infty]|_{D_\infty} \cong\mathbf{L}^{-1}|_{D_\infty}$, thus $\tilde{\pi}^*(\mathbf{F})|_{D_\infty}$
is trivial.  The projection $D_\infty =\mathbb{P}(\mathbf{E})\rightarrow N$ induces an injection on the Picard group,
therefore $[D_\infty]=\mathbf{L}^{-1}$.

In particular, suppose $\pi:\mathbf{E}\rightarrow N$ is a negative bundle.  Further, suppose that if $M'$ denotes the
total space of $\mathbf{E}$, $c_1(M')<0$.  This can be seen to be equivalent to $c_1(N)+c_1(\mathbf{E})<0$.
For if $-\varpi \in c_1(N)+c_1(\mathbf{E})$ with $-\varpi$ negative, then
$-\pi^*\varpi -\sqrt{-1}\partial\ol{\partial}r^2$ is a negative form in $c_1(M')$ where $r^2 =h(v,v)$.
By (\ref{eq:can-bund2}), Proposition~\ref{prop:neg-taut}, and the above comments we have the following.
\begin{prop}
Let $\pi:\mathbf{E}\rightarrow N$ be a negative bundle of rank r with $c_1(N)+c_1(\mathbf{E})<0$.  The canonical bundle
of $X=\mathbb{P}(\mathbf{E}\oplus\C)$ satisfies
\begin{equation}
\mathbf{K}_X \otimes [kD_\infty]>0,\quad\text{for } k>r+1.
\end{equation}
\end{prop}

Proposition~\ref{prop:quasi-proj} then gives the following existence result for Einstein metrics on negative bundles.
\begin{cor}\label{cor:neg-bund}
Let $M'$ be the total space of a negative holomorphic bundle $\pi:\mathbf{E}\rightarrow N$ such that $c_1 (M')<0$, equivalently
$c_1(N) +c_1(\mathbf{E})<0$, then the strictly pseudoconvex tubular neighborhoods $M_c =\{v\in\mathbf{E}: r^2=h(v,v)<c\}$ of
the zero section admit unique complete K\"{a}hler-Einstein metrics, with normal CR infinity the sphere bundle
$S_c \subset\mathbf{E}$.
\end{cor}

\subsection{Resolutions weighted homogeneous hypersurfaces}

One can easily construct examples by taking weighted blow-ups of simple weighted homogeneous hypersurface singularities.

A polynomial $f\in\C[z_0\ldots,z_n]$ is \emph{weighted homogeneous} with weights $\mathbf{w}=(w_0,\ldots,w_n)\in\Z_+^{n+1}$ and
degree $d$ if
\begin{equation}
f(u^{w_0} z_0,\ldots,u^{w_n} z_n) =u^d f(z_0,\ldots,z_n),\quad u,z_0,\ldots,z_n \in\C.
\end{equation}
Here we assume that $\gcd(w_0,\ldots,w_n)=1$.  We assume that $X=\{z\in\C^{n+1} : f(z)=0\}$ is smooth away from
$o\in\C^n$.  Then it is well known, see~\cite{BoyGalKol}, that $S=X\cap\mathbb{S}^{2n+1}$ has a natural Sasaki structure with Reeb vector field generating the action  $(z_0,\ldots,z_n)\rightarrow (u^{w_0} z_0,\ldots,u^{w_n} z_n)$ and the CR structure of $S$ satisfies property S-E
precisely when $|\mathbf{w}|=\sum w_i >d$, \textit{loc. cit.}.
The codimension of the singular set of $X$ is $\geq 2$, so $X$ is normal.
And $X$ is easily seen to be Gorenstein with holomorphic form given on $X\setminus\{o\}$ by adjunction by
\begin{equation}
\Omega =\frac{(-1)^{i+1}}{\partial f/\partial z_i} dz_0 \wedge\cdots\wedge\hat{dz_i}\wedge\cdots\wedge dz_n,
\end{equation}
where $\frac{\partial f}{\partial z_i} \neq 0$ for $i=1,\ldots, n$.
Therefore we have the following.
\begin{prop}
A weighted homogeneous hypersurface $X=\{z\in\C^{n+1} : f(z)=0\}$ with an isolated singularity has a resolution
$\hat{X}$ with $\mathbf{K}_{\hat{X}} >0$ only if $\sum w_i \leq d=\deg(f)$.
\end{prop}

We have $C(S)_{\leq 1} =X\cap\mathbb{B}^{n+1}$.  And by Proposition~\ref{prop:nor-st-pseudo} any strictly pseudoconvex domain with
CR infinity $S$ must be the domain $\pi^{-1}(X\cap\mathbb{B}^{n+1})$ in a resolution $\pi:\hat{X}\rightarrow X$.
Examples of such resolutions are easy to find by taking blow-ups or more generally weighted blow-ups.
The weight $\mathbf{w}=(w_0,\ldots,w_n)$ defines a grading $\C[z_0,\ldots,z_n]=\oplus_{k\geq 0} \C[z_0,\ldots,z_n]_k$.
We define $\deg(f)=\max\{j: f\in\oplus_{k\geq j} \C[z_0,\ldots,z_n]_k \}$.
The \emph{weighted blow-up} $\varpi: B^{\mathbf{w}} \C^{n+1} \rightarrow\C^{n+1}$ is constructed similarly to the usual but with the weighted
grading, and the exceptional fiber $E=\varpi^{-1}(o)=\mathbb{P}(w_0,\ldots,w_n)$, the weighted projective space.
And, of course, one obtains the usual blow-up with $\mathbf{w}=(1,\ldots,1)$.  If $X'\subset B^{\mathbf{w}} \C^{n+1}$ is the
strict transform, then we have the adjunction formula for the canonical bundle
\begin{equation}\label{eq:adjunct}
\mathbf{K}_{X'} =\varpi^*(\mathbf{K}_X) +(|\mathbf{w}|-\deg(f)-1)X'\cap E,
\end{equation}
provided $X'$ does not contain a divisor singular along a singular set of $B^{\mathbf{w}} \C^{n+1}$.  See~\cite{Reid} for more details.

\subsubsection{Example 1}

Consider the hypersurface
\begin{equation}
X=\{z_0^d +\cdots +z_{n-1}^d +z_n^k =0\}\subset\C^{n+1},
\end{equation}
with $k\geq d\geq n+1$.  We consider a series of blow-ups of $X$.  Blowing up gives $\pi: X_1 \rightarrow X$ where $X_1$
is the strict transform of $X$ in $\pi:\hat{\C}^{n+1}\rightarrow\C^{n+1}$, the blow-up of $\C^{n+1}$ at the origin.
Then $X_1$ is covered with affine neighborhoods $U_i ,i=0,\ldots, n$.  Take for example $U_0 \subset\C^{n+1}$ which has
coordinates $y_0,\ldots, y_{n}$ and $\pi$ is given by $z_0 =y_0 , z_1 =y_0 y_1 ,\ldots, z_{n} =y_0 y_{n}$.
Thus if $f=z_0^d +\cdots +z_{n-1}^d +z_n^k$, then $\pi^* f=y_0^d(1+y_1^d +\cdots+y_{n-1}^d +y_0^{k-d}y_n^k)$.
So if $g=1+y_1^d +\cdots+y_{n-1}^d +y_0^{k-d}y_n^k$, then $X_1 \cap U_0 =\{g=0\}\subset\C^{n+1}$.
It is elementary to check that this is a non-singular hypersurface and similarly for $X_1 \cap U_i, i=1,\ldots,n-1$.
We have $X_1 \cap U_n =\{g=0\}\subset\C^{n+1}$ where $g=y_0^d +\cdots+y_{n-1}^d +y_n^{k-d}$, and this hypersurface has a singular
point at the origin unless $k-d=0$ or $1$.  Repeating the procedure we get a resolution $\pi:\hat{X}\rightarrow X$,
$\hat{X} =X_{\lfloor\frac{k}{d}\rfloor}$, if $k \equiv 0$ or $1 \mod d$.

Denote by $E_i$ the strict transform of the exceptional set of the $i-th$ blow-up.  Then if follows from (\ref{eq:adjunct}) that
\begin{equation}
\mathbf{K}_{\hat{X}} =\sum_{i=1}^{\lfloor\frac{k}{d}\rfloor} i(n-d)E_i.
\end{equation}
In order to prove that $\mathbf{K}_{\hat{X}} >0$ we will compactify $\hat{X}$ and employ a lemma of H. Grauert.
Let $s=\lcm(d,k)$ and set $a=\frac{s}{d}$ and $b=\frac{s}{k}$.  Then we have $\C^{n+1}\subset\cps^{n+1}_{a,\ldots, a,b,1}$, where
$\cps^{n+1}_{a,\ldots, a,b,1}$ is the weighted projective space, and $f= z_0^d +\cdots +z_{n-1}^d +z_n^k$ is weighted homogeneous
with respect to these weights.  Let $Y=\{f=0\}\subset\cps^{n+1}_{a,\ldots, a,b,1}$.  If $z_0, \ldots, z_{n+1}$ are homogeneous
coordinates on $\cps^{n+1}_{a,\ldots, a,b,1}$, then we have added $\{z_{n+1}=0\}$ to $\C^{n+1}$.
Let $E_{\infty} = Y\cap\{z_{n+1}=0\}$.  Let $\hat{Y}$ be the above resolution of $Y$ given by resolving $X\subset Y$.
We will prove that
\begin{equation}
\mathbf{F}=\sum_{i=1}^{\lfloor\frac{k}{d}\rfloor} i(n-d)[E_i] +t[E_\infty] >0
\end{equation}
on $\hat{Y}$ for $t\in\N$ sufficiently large.  Note that $\mathbf{F}$ is not a Cartier divisor unless $\lcm(a,b) |t$.
We will use the following due to H. Grauert~\cite{Grau}.
\begin{lem}\label{lem:Grauert}
A line bundle $\mathbf{L}$ on a compact complex space $X$ is positive if and only if for every irreducible compact nowhere discrete
analytic subspace $Z\subset X$ there is an holomorphic section $\sigma$ of $\mathbf{L^{k}}|_Z$, for some $k$, with a zero on $Z$ but not vanishing entirely.
\end{lem}
Let $H=\{f=0\}\subset\C^{n+1}$ be an hypersurface which is tangent to the line $\C (0,\ldots,0,1)$ at $(0,\ldots,0)\in\C^{n+1}$ to at least order $\lfloor\frac{k}{d}\rfloor -1$.  If $D=H\cap X$, then one can check that $\pi^* D=D' +\sum_{i=1}^{\lfloor\frac{k}{d}\rfloor} iE_i$,
where $D'$ is the strict transform of $D$.  Assume that $f$ is algebraic, so it extends to a rational function on $Y$ with a pole
along $E_{\infty}$.  If $t>0$ is sufficiently large, then $\pi^* f^{(n-d)}$ gives an holomorphic section of $\mathbf{F}$ with
$(\pi^* f^{(n-d)} )=D' +qE_{\infty}$, for some $q>0$.  It is not very difficult to check that, for various such $D$, the condition
of Lemma~\ref{lem:Grauert} is satisfied for $A$ not contained in $E_{\infty}$.

Suppose $A\subset E_{\infty}$.  The above argument gives a section $\sigma$ of $\mathbf{F}^k$ with $(\sigma)=kD' +kqE_{\infty}$.
If $kq$ is sufficiently divisible by $a$ and $b$, then there are many rational functions $\frac{g}{z_{n+1}^{kq}}$ with $g$ not
vanishing along $E_{\infty}$.  We have $(\sigma)\sim kD' +(g)$, and various choices of $D$ and $g$ give the required section of
$\mathbf{F}^k |_A$.

\subsubsection{Example 2}

Let
\begin{equation}
X=\{z_0^d +z_1^{2d} +\cdots +z_{n-1}^{2d} +z_n^k =0\}\subset\C^{n+1},
\end{equation}
with $k\geq 2d\geq n+2$.  Then $X$ can be resolved similar to Example 1 but by taking blow-ups with weight
$\mathbf{w}=(2,1,\ldots, 1)$.  One can repeatedly blowing up the unique singular with this weight $\lfloor\frac{k}{2d}\rfloor$
times.  And if $k\equiv 0$ or $1 \mod 2d$, this ends in a smooth resolution $\pi:\hat{X}\rightarrow X,\ X_{\lfloor\frac{k}{2d}\rfloor}$.  One can check using arguments as in Example 1 that $\mathbf{K}_{\hat{X}} >0$.

\subsubsection{Example 3}

Let
\begin{equation}
X=\{z_0^{2d} +z_1^{3d} +z_2^{6d} +\cdots +z_{n-1}^{6d} +z_n^k \},
\end{equation}
with $k\geq 6d\geq n+4$.  Again $X$ can repeatedly blowing up with weight $\mathbf{w}=(3,2,1,\ldots,1)$ the
unique singular point at each step $\lfloor\frac{k}{6d}\rfloor$ times.  If $k\equiv 0$ or $1 \mod 6d$, this ends in a smooth resolution $\pi:\hat{X}\rightarrow X,\ X_{\lfloor\frac{k}{6d}\rfloor}$.  And again similar arguments show that
$\mathbf{K}_{\hat{X}} >0$.

\subsection{Locally strongly ACH K\"{a}hler-Einstein manifolds}

We will consider examples of $ACH$ K\"{a}hler manifolds where the metric converges to that of the complex hyperbolic space
$\mathcal{H}^n_{\C}$.  More precisely, suppose $\Gamma\subset\operatorname{PSU}(1,n)$ is a finite group.  By making a conjugation,
we may assume that $\Gamma\subset\operatorname{U}(n)\subset\operatorname{PSU}(1,n)$.  Assume that $\Gamma$ acts freely
away from $o\in\mathbb{B}^n =\mathcal{H}^n_{\C}$, where $\mathbb{B}^n \subset\C^n$ is the unit ball.

\begin{defn}
The K\"{a}hler manifold $(M,g)$ is \emph{locally strongly ACH}, of order $\alpha>0$, if there is a compact set $K\subset M$, a ball
$B\subset\mathcal{H}^n_{\C}$, and a biholomorphism $\psi :M\setminus K \rightarrow \mathcal{H}^n_{\C} \setminus B/\Gamma$,
such that if $g_0$ is the hyperbolic metric of $\mathcal{H}^n_{\C} \setminus B/\Gamma$,
\begin{equation}\label{eq:met-conv}
|\psi_* g -g_0|_{g_0} =O\bigl(e^{-\alpha r}\bigr),
\end{equation}
where $r=\dist(o,x)$ is the distance from a fixed point of $\mathcal{H}^n_{\C}/\Gamma$.
\end{defn}
\begin{remark}
One generally also has to assume an analogous condition to (\ref{eq:met-conv}) on some derivatives of $g$ for most analytical
purposes, i.e.
\begin{equation}\label{eq:der-conv}
|\nabla^k\bigl(\psi_* g -g_0 \bigr)|_{g_0} =O\bigl(e^{-\alpha r}\bigr),
\end{equation}
where $\nabla^k$ is the k-th covariant derivative of $g_0$.

Also, one may consider the weaker condition that $\psi$ is merely a diffeomorphism and the complex structure $J$ of $M$
converges to $J_0$ of $\mathcal{H}^n_{\C}/\Gamma$ as in (\ref{eq:met-conv}).  See~\cite{Herz,BouHerz} for some interesting rigidity results
for such strongly ACH manifolds.  Especially considering Proposition~\ref{prop:st-ACH}, it is an interesting question whether there
are similar rigidity results for \emph{locally} strongly ACH manifolds.
\end{remark}

\begin{prop}\label{prop:st-ACH}
Let $(M,g)$ be a K\"{a}hler manifold with K\"{a}hler form $\omega\in-\frac{2\pi}{n+1}c_1(M)$, and assume that $M\setminus K$
is biholomorphic to $\mathbb{B}^n \setminus B/\Gamma$, with $\Gamma$ as above.  Then there is a locally strongly ACH
K\"{a}hler-Einstein metric $g$ on $M$ which is of order $\alpha$ for all $\alpha <2n+2$, with convergence
including all derivatives.  That is (\ref{eq:der-conv}) holds for all $k\geq 0$.
\end{prop}
\begin{proof}
By Proposition~\ref{prop:nor-st-pseudo} $M$ is a domain in a resolution $\pi: X\rightarrow\C^n /\Gamma$.
Since this a resolution of a rational singularity the $\partial\ol{\partial}$-lemma holds.  Thus $\mathbf{K}_X$  admits an Hermitian
metric $h$ so that $\omega =\frac{\sqrt{-1}}{n+1}\Theta_h$ is a K\"{a}hler form.

If $k=|\Gamma|$, then $(dz_1 \wedge\cdots\wedge dz_n)^{\otimes k} \in\Gamma\bigl(\mathbf{K}^k_{\C^n/\Gamma} \bigr)$
and $\sigma =\pi^* (dz_1 \wedge\cdots\wedge dz_n)^{\otimes k}$ is a meromorphic section of $\mathbf{K}^k_X$, which incidentally must
have poles on the exceptional divisors.  We have a metric on $\mathbf{K}^k_X$, also denoted by $h$, with
$\omega =\frac{\sqrt{-1}}{k(n+1)}\Theta_h$.  Let $r^2 =\sum_j |z_j|^2$, and for $0<\epsilon<\frac{1}{2}$ define a cut-off function
$0\leq\rho(r)\leq 1$
\[ \rho(r)=\begin{cases}
1, & r<\epsilon \\
0, & r>2\epsilon\\
\end{cases}\]
Let $\tilde{h}$ be the metric on $\mathbf{K}^k_{X\setminus E}$, where $E$ is the exceptional set, with
$\tilde{h}|\sigma|^2 =1$.  Define the metric
\begin{equation}
\hat{h}:= e^{-Cr^2}\bigl(\rho(r) h +(1-\rho(r))\tilde{h}\bigr),
\end{equation}
which for $C>0$ sufficiently large $\omega_0 :=\frac{\sqrt{-1}}{k(n+1)}\Theta_{\hat{h}}$ is positive.
Since $\omega_0^n$ also defines an Hermitian metric on $\mathbf{K}_X^k$, we can define $f\in C^\infty(X)$ by
\begin{equation}\label{eq:equiv-Her}
e^{kf} \hat{h} =\frac{1}{\bigl( \omega_0^n \bigr)^k},
\end{equation}
where $f=c +Cr^2$ for $r>2\epsilon$.  In particular, we have (\ref{eq:omega0-Ricci}).  The defining function
$\phi =e^{\frac{f}{n+1}}\phi_0$ where $\phi_0 =\bigl[\frac{C}{n+1} \bigr]^{\frac{n}{n+1}} \bigl(r^2 -1 \bigr)$ has $-dd^c \log(-\phi)>0$
in a neighborhood of $\partial\mathbb{B}^n$.  We may modify $\phi$ on $r<\varepsilon<1$ so that $-dd^c \log(-\phi)\geq 0$ on
$M$.  Then on $r>\max(2\epsilon,\varepsilon)$ (\ref{eq:F}) becomes
\begin{equation}\label{eq:F-xpl}
F=\log\left[\frac{(-\phi_0)^{-(n+1)}\omega_0^n}{(-dd^c \log(-\phi_0))^n}  \right]=0,
\end{equation}
as is not difficult to check.  Therefore, when we solve (\ref{eq:M-A}) to get the K\"{a}hler-Einstein metric
$\omega'$ on $r>\max(2\epsilon,\varepsilon)$ we have
\begin{equation}\label{eq:u-xpl}
\begin{split}
\omega' & =\omega_0 -dd^c \log (-\phi) + dd^c u \\
        &  =-dd^c \log (-\phi_0) + dd^c u \\
        &  =-dd^c \log (1 -r^2) + dd^c u, \\
\end{split}
\end{equation}
where $u\in C^{\infty}(M)$ satisfies both (\ref{eq:u-asym-hol}) and (\ref{eq:u-asym}).  The conclusion follows
from the observation that if $r=\dist(o,x)$ is the distance with respect to the Bergman metric and $\phi$ is any defining function,
then $c(-\phi)\leq e^{-2r} \leq C(-\phi)$ for $C>c>0$.
\end{proof}

\subsubsection{Resolutions of Hirzebruch-Jung singularities}\label{subsect:Hirz-Jung}

Let $p>q>0$ be relatively prime integers and consider the finite group $\Gamma\subset U(2)$ generated by
\begin{equation}
\begin{bmatrix}
e^{\frac{2\pi iq}{p}} & 0 \\
0  & e^{\frac{2\pi i}{p}}
\end{bmatrix}.
\end{equation}
Then $\C^2 /\Gamma$ has an isolated orbifold singularity at the origin, and its minimal resolution given by a
Hirzebruch-Jung string is well known.  See~\cite[Ch. II, \S 5]{BarHulPetVan} and~\cite[\S 1.6]{Oda} for a description in terms
of toric geometry.  This minimal resolution $\pi: X\rightarrow\C^2/\Gamma$ has the following properties:
\begin{thmlist}
\item   The exceptional divisor $E=\pi^{-1}(0)=\cup_{i=1}^k C_i$, where each $C_i$ is an embedded $\cps^1$.
\item   $C_i ^2 =-e_i \leq -2$, $C_i \cdot C_j =1$ for $|i-j|=1$, and $C_i \cdot C_j =0$ for $|i-j|>1$.
\end{thmlist}

The integers $e_i$ are given by the continued fraction expansion
\begin{equation}
\frac{p}{q}=e_1 -\cfrac{1}{e_2 -\cfrac{1}{e_3 -\cfrac{1}{\ddots -\cfrac{1}{e_k}}}}
\end{equation}
In other words, the $e_i$ are determined by the Euclidean algorithm where we define $q_i, -1\leq i\leq k$, inductively
\begin{equation}
q_{-1} := p,\quad q_{0}:= q,\quad q_{i-1} =e_{i+1}q_i -q_{i+1},\quad\text{with } 0\leq q_{i+1}<q_i .
\end{equation}
Note that the resolution $\pi: X\rightarrow\C^2/\Gamma$ is the unique minimal toric resolution.  And
one can retrieve the toric diagram, i.e. the stabilizers of the $C_i$, from the $e_i$ as follows.  Suppose
$C_i$ has stabilizer $(m_i ,n_i)\in\Z^2$, with $m_i$ and $n_i$ coprime, then we have
\begin{equation}\label{eq:stabil}
\frac{n_i}{n_i -m_i}=e_1 -\cfrac{1}{e_2 -\cfrac{1}{e_3 -\cfrac{1}{\ddots -\cfrac{1}{e_{i-1}}}}}.
\end{equation}
And from (\ref{eq:stabil}) we obtain $(m_i,n_i)$ unambiguously since $m_i>0$.  If we denote by $C_0$ and $C_{k+1}$ the
non-compact curves with stabilizers $(1,0)$ and $(p-q,p)$ respectively corresponding to the axis of $\C^2/\Gamma$, then
the arrangement of curves is given in Figure~\ref{fig:resol}.

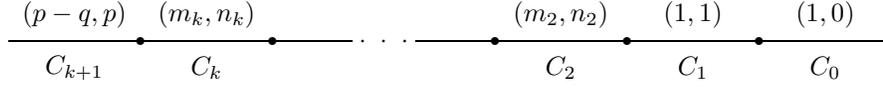
\begin{figure}\label{fig:resol}
\setlength{\unitlength}{.5pt}
\centering
\begin{picture}(700,60)(0,-30)

\put(0,0){\line(1,0){260}}
\put(268,0){\circle*{1}}
\put(284,0){\circle*{1}}
\put(300,0){\circle*{1}}

\put(308,0){\line(1,0){360}}

\put(100,0){\circle*{5}}
\put(200,0){\circle*{5}}

\put(368,0){\circle*{5}}
\put(468,0){\circle*{5}}
\put(568,0){\circle*{5}}

\put(50,-20){\makebox(0,0){$C_{k+1}$}}
\put(150,-20){\makebox(0,0){$C_{k}$}}
\put(418,-20){\makebox(0,0){$C_{2}$}}
\put(518,-20){\makebox(0,0){$C_{1}$}}
\put(618,-20){\makebox(0,0){$C_{0}$}}

\put(50,20){\makebox(0,0){$(p-q,p)$}}
\put(150,20){\makebox(0,0){$(m_k,n_k)$}}
\put(418,20){\makebox(0,0){$(m_2,n_2)$}}
\put(518,20){\makebox(0,0){$(1,1)$}}
\put(618,20){\makebox(0,0){$(1,0)$}}
\end{picture}
\caption{resolution of $\C^2/\Gamma$}
\end{figure}

We have $c_1(X)<0$ if and only if $e_i \geq 3$ for $0\leq i\leq k$, in which case Proposition \ref{prop:st-ACH}
gives a locally strongly ACH K\"{a}hler-Einstein metric on the domain $M_a =\{r<a\}\subset X$, where $r^2 =|z_1|^2 +|z_2|^2$,
which is invariant under $T^2 \subset\operatorname{U}(2)$.

It is known from D. Calderbank and M. Singer~\cite{CalSin} that the domains $M_a$ admit ASD (Self-dual Weyl curvature vanishes,
$W_+ =0$.) Hermitian Einstein metrics with negative scalar curvature.  These metrics are toric but not K\"{a}hler.  And furthermore,
these metrics have the same CR-infinity as the K\"{a}hler-Einstein metrics given here.  It would be interesting to know how these
metrics are related in say the moduli of Einstein structures on $M_a$.

\subsubsection{Resolutions of diagonal quotients $\C^n /\Gamma$}\label{subsect:diag-quot}

Let $\Gamma =\Z_k$ be the group acting on $\C^n$ with generator $(z_1,\ldots,z_n) \mapsto (\zeta\cdot z_1,\ldots,\zeta\cdot z_n)$,
$\zeta=e^{\frac{2\pi i}{k}}$.  Then the singularity $\C^n /\Gamma$ has as resolution $\pi:X\rightarrow\C^n /\Gamma$ the total space of
$\mathbf{L}^k \rightarrow\cps^{n-1}$, with $\mathbf{L}$ the tautological line bundle.  Then $c_1(X)<0$, if and only if $k>n$,
and in this case Proposition~\ref{prop:st-ACH}
gives a locally strongly ACH K\"{a}hler-Einstein metric on the domain $M_a =\{r<a\}\subset X$, where $r^2 =\sum_j |z_j|^2$,
which is invariant under $\operatorname{U}(n)$.

The metrics in this example are of cohomogeneity one, so one expects to find an explicit formula for the K\"{a}hler potential
of $\omega$.  There is a simple formula for these metrics due to A. Futaki~\cite{Fut}.   The 2-form $dd^c \log r^2$
is basic with respect to the $\C^*$-action on $\C^n/\Gamma$ and restricts to the Fubini Study metric on the quotient $\cps^{n-1}$.
Define
\begin{equation}
\omega^T =kdd^c \log r^2.
\end{equation}
Then the transversal Ricci form, i.e. that of $\omega^T$ on $\cps^{n-1}$, satisfies
\begin{equation}
\ric^T =\frac{n}{k}\omega^T.
\end{equation}
Let $t=\log r^{2k}$.  The technique of E. Calabi~\cite{Cala} is to consider metrics of the form
\begin{equation}\label{eq:Calabi}
\omega=\omega^T +\sqrt{-1}dd^c F(t).
\end{equation}
It turns out to be easier to work in a \emph{momentum coordinate} along the fiber.  So set
\begin{gather}
\tau= F'(t), \\\
\phi(\tau) = F''(t).
\end{gather}
Then (\ref{eq:Calabi}) becomes
\begin{equation}
\begin{split}
\omega & =(1+\tau)\omega^T +\phi(\tau)dt\wedge d^c t\\
       & =(1+\tau)\omega^T +\phi(\tau)^{-1} d\tau\wedge d^c \tau.
\end{split}
\end{equation}
Then its Ricci form is computed in~\cite{Fut} to be
\begin{equation}
\ric(\omega) =\ric^T -dd^c \log\bigl( (1+\tau)^{n-1} \phi(\tau)\bigr),
\end{equation}
and the Einstein equation $\ric(\omega) =-\lambda\omega$ is satisfied with $\lambda =2-\frac{n}{k}$ if
\begin{equation}\label{eq:profile}
\phi(\tau)=\frac{1}{k}(1+\tau) -\frac{(n-2k)}{k(n+1)}(1+\tau)^2 -\frac{(2n+k)}{n(n+1)}\frac{1}{(1+\tau)^{n-1}}.
\end{equation}
One retrieves the complex coordinate expression (\ref{eq:Calabi}) by integrating, for fixed $\tau_0$,
\begin{equation}\label{eq:moment-inv}
t=\int_{\tau_0}^{\tau(t)} \frac{dx}{\phi(x)}.
\end{equation}
And one also obtains
\begin{equation}
F(t) =\int_{\tau_0}^{\tau(t)} \frac{x dx}{\phi(x)}.
\end{equation}
Since (\ref{eq:profile}) grows quadratically (\ref{eq:moment-inv}) shows that the range of $t$ is finite.  In fact,
$F'(t)$ maps $(-\infty, c)$ to $(0,\infty)$.  And one can show that it is a complete metric defined on $\{r<e^{\frac{c}{2k}}\}\subset X$. The differing radii arise from the ambiguity in the integral (\ref{eq:moment-inv}).

It would be interesting to obtain a closed formula for the K\"{a}hler potential of the metric (\ref{eq:Calabi}).
There is such a formula for the Ricci-flat metric in case $k=n$ on $X$ due to E. Calabi~\cite[4.14]{Cala}.
Since the explicit formula involves integrals it is easier to see the strongly ACH nature of the metric from Proposition~\ref{prop:st-ACH}.

\subsection{Normal CR infinities in dimension 3}

Using the classification of normal CR structures on 3-manifolds in~\cite{Belgun2} and~\cite{Belgun3} we are able to mostly classify
those normal CR 3-manifolds which bound K\"{a}hler-Einstein manifolds and the unique K\"{a}hler-Einstein surfaces which thus arise.

The classification of normal CR structures on 3-manifolds follows from a classification of Sasaki structures on 3-manifolds
which in turn follows from a classification~\cite{Belgun1} of Vaisman metrics (or locally comformally K\"{a}hler metrics with a parallel
Lee forms) on compact surfaces.  The Riemannian product of a Sasaki manifold with a circle is a Vaisman manifold.
\begin{thm}[\cite{Belgun2}]\label{thm:Sasak-class}
If $(S,g,\xi)$ is a Sasaki 3-manifold then it is one of the followning.
\begin{thmlist}
\item  $S$ is a Seifert $\mathbb{S}^1$-bundle over a Riemann surface of genus $g>1$, $\xi$ generates the $\mathbb{S}^1$-action,
and the Vaisman manifold $S\times\mathbb{S}^1$ is a properly elliptic surface admitting two holomorphic circle actions.\label{it:pr-ellip}

\item  $S$ is a Seifert $\mathbb{S}^1$-bundle over an elliptic curve, $\xi$ generates the $\mathbb{S}^1$-action, and the Vaisman manifold
$S\times\mathbb{S}^1$ is a Kodaira surface admitting two holomorphic circle actions.\label{it:kod}

\item  $S$ is a finite quotient of $\mathbb{S}^3$, with holomorphic coordinates $(z_1 ,Z_2)=(x_1 +iy_1,x_2 +iy_2)$ we have
$\xi=a(x_1 \partial_{y_1}-y_1 \partial_{x_1}) +b(x_2 \partial_{y_2} -y_2 \partial_{x_2})$ with $a\geq b>0$, and the Vaisman manifold
$S\times\mathbb{S}^1 =\C^2 \setminus\{(0,0)\}/G$ is a Hopf surface of class 1 where $G$ is generated by the contraction
$g(z_1,z_2)=(e^{-a} z_1, e^{-b}z_2)$.\label{it:Hopf}
\end{thmlist}
\end{thm}
The classification of normal CR manifolds $(S,D,J)$ is more complicated as as it involves identifying and distinguishing the underlying
CR structures of the above Sasaki structures.  This is solved for (\ref{it:pr-ellip}) and (\ref{it:kod}) by showing there are no other CR Reeb
vector fields, so other Sasaki structures are deformations as in Section~\ref{subsubsect:transv}.

An above $\mathbb{S}^1$-Seifert bundle is an $\mathbb{S}^1$ subbundle of a negative orbifold bundle $\mathbf{L}$ over a Riemann surface
$N$.  Suppose it has multiple fibers over $p_1,\ldots,p_k \in N$ of multiplicities $m_1,\ldots, m_k$.  We can classify the multiple
fibers by $(m_j ; q_j),\ q_j <m_j,\ j=1,\ldots,k,$, where there is a neighborhood $U$ of $p_j$ and locally $N$ is the quotient of $U\times\mathbb{S}^1$ by the $\Z_{m_j}$-action generated by $(z,w)\mapsto (e^{2\pi i/m_j} z,e^{2\pi iq_j/m_j} w)$.

\begin{prop}
Suppose $(S,D,J)$ is a normal CR 3-manifold.
\begin{thmlist}
\item  If $b_1(S)>0$, that is cases (\ref{it:pr-ellip}) and (\ref{it:kod}) in Theorem~\ref{thm:Sasak-class}, then $(S,D,J)$ is the
boundary of an unique K\"{a}hler-Einstein manifold if and only if for each $(m_j ; q_j)\ j=1,\ldots,k,$ the $e_i$ produced by the Euclidean
algorithm of Section~\ref{subsect:Hirz-Jung} satisfy $e_i \geq 3$. \label{it:quasi-reg}

\item  If $b_1(S)=0$, case (\ref{it:Hopf}), then we have a biholomorphism $C(S)\cong\C^2/\Gamma$, with $\Gamma\subset\operatorname{GL}(2,\C)$
finite and acting freely on $\C^2 \setminus\{(0,0)\}$.  Then $(0,0) \in\C^2/\Gamma$ is a rational singularity, thus the exceptional set of the minimal resolution $\pi:X\rightarrow \C^2/\Gamma$ is a tree of rational curves $C_i,\ i=1,\ldots,k$.  This is the unique K\"{a}hler-Einstein
manifold with boundary $(S,D,J)$ if and only if $C_i^2 \leq -3,\ i=1,\ldots,k$.\label{it:sphere}
\end{thmlist}
\end{prop}
\begin{proof}
For case (\ref{it:quasi-reg}) $S$ is a Seifert $\mathbb{S}^1$ subbundle of a negative orbifold bundle $\mathbf{L}$ over a Riemann surface
$N$ of genus $g>0$.  Either we have a regular Sasaki structure and $\mathbf{L}$ is smooth, or there are multiple fibers and corresponding
singularities $(m_j ; q_j)\ j=1,\ldots,k.$  Each of these can be resolved as in Section~\ref{subsect:Hirz-Jung}, altogether
giving the unique minimal resolution $\pi: X\rightarrow C(S)$.  It is clear that $e_i =-C^2_i\geq 3$ is necessary from the adjunction
formula $2g(C_i) -2 =-2 =K_X \cdot C_i +C_i^2$.  If $C_0$ denotes the zero section of $\mathbf{L}$, then
\begin{equation}
0\leq 2g(C_0) -2 =K_X \cdot C_0 +C_0^2.
\end{equation}
By Grauert's criterion for an exceptional curve~\cite[{p. 367}]{Grau} $C_0^2 <0$, so $K_X \cdot C_0 >0$.
Therefore by assumption $K_X \cdot C >0$ for each irreducible exceptional curve.  And if $E=\cup_i C_i$ denotes
the exceptional set, the argument in the proof of the lemma on p. 347 of~\cite{Grau} shows that
$\mathbf{K}_X |_E >0$.

We have $C(S)=\mathbf{L}^\times$, minus the zero section.  Then, after a possible homothetic change of $\xi$, the radial $r$ on $C(S)$ is
given by $r^2 =h(v,v)$ for an Hermitian metric $h$ on $\mathbf{L}$.  Let $\omega\in\Gamma(\mathbf{K}_N)$ be an holomorphic section which
vanishes on $x_1,\ldots, x_{2g-2}$.  Define the $(1,0)$-form on the total space of $\mathbf{L}$
\begin{equation}
\beta =J^* \eta +\sqrt{-1}\eta =\frac{dr}{r}+\sqrt{-1}\eta =\partial\log r^2.
\end{equation}
Then the $(2,0)$-form $\Omega :=\beta\wedge\pi^* \omega$ satisfies $\ol{\partial}\Omega=0$ and has a pole of order 1 on the zero section.
Strictly speaking, we need to take this local construction to the power $\ell=\lcm(m_1,\ldots,m_k)$ to get a true meromorphic section
$\Omega^\ell \in\Gamma(\mathbf{K}_X^\ell)$.  So $\Omega^\ell$ is a meromorphic section with zeros along $\pi^*(x_i),\ i=1,\ldots 2g-2,$
and poles of varying orders on the exceptional curves $C_i$.  Then the arguments in Example 1 above or in the proof of
Satz 4 in~\cite[{p. 367}]{Grau} show that $\mathbf{K}_X >0$.

In case (\ref{it:sphere}), $C(S)\cong\C^2/\Gamma$ follows from Theorem~\ref{thm:Sasak-class}.\ref{it:Hopf}.
It is well known that $\C^2/\Gamma$ has only rational singularities, and the properties of the exceptional curves follow from
well known properties of rational surface singularities (cf.~\cite{BarHulPetVan}).  If the exceptional curves
satisfy $C_i^2 \leq -3$ the arguments in part (\ref{it:quasi-reg}) show what $\mathbf{K}_X >0$.
\end{proof}

In part (\ref{it:sphere}) the groups $\Gamma\subset\operatorname{GL}(2,\C)$ are classified~\cite{NguPutTop}.  This follows from
\begin{equation}
1\rightarrow\C^* \rightarrow \operatorname{GL}(2,\C)\rightarrow\operatorname{PGL}(2,\C)\rightarrow 1,
\end{equation}
and the fact that the finite subgroups of $\operatorname{PGL}(2,\C)$ are the polyhedral groups.

We already know which cyclic groups $\Z_p =\Gamma\subset\operatorname{GL}(2,\C)$ have the require resolution from Section
~\ref{subsect:Hirz-Jung}.  And all the groups $\Gamma\subset\operatorname{SL}(2,\C)$ are ruled out by Theorem~\ref{thm:poss-inf}.

\bibliographystyle{plain}

\end{document}